%% file: Paper.tex
\documentclass[11pt,oneside,reqno]{amsart}

\usepackage[utf8]{inputenc}
\usepackage[margin=2cm]{geometry}
\usepackage{amsmath,amsfonts,amsthm,amssymb}
\usepackage{graphicx} 
\usepackage[backend=biber,style=trad-abbrv]{biblatex}
\usepackage{xcolor}
\usepackage{tikz}
\usetikzlibrary{positioning}

\usepackage{hyperref}
\hypersetup{
  colorlinks   = true, 
  urlcolor     = blue, 
  linkcolor    = red, 
  citecolor   = red 
}
\usepackage{bbm}
\usepackage{comment}
\usepackage{enumerate}

\usepackage[normalem]{ulem}

\newcommand\CL{\mathcal{L}}

\newcommand\CF{\mathcal{F}}

\newcommand\Lie{\text{Lie}}

\newcommand\IE{\mathbb{E}}
\newcommand\IP{\mathbb{P}}
\newcommand\IN{\mathbb{N}}

\newcommand\IR{\mathbb{R}}

\newcommand\df{\textup{d}}

\newcommand\W{\widetilde{W}}
\newtheorem{theorem}{Theorem}[section]
\theoremstyle{plain}
\newtheorem{proposition}[theorem]{Proposition}
\newtheorem{assumption}[theorem]{Assumption}
\theoremstyle{definition}

\theoremstyle{definition}
\newtheorem{remark}[theorem]{Remark}
\theoremstyle{plain}
\newtheorem{lemma}[theorem]{Lemma}
\theoremstyle{definition}
\newtheorem{definition}[theorem]{Definition}

\theoremstyle{plain}

\numberwithin{equation}{section}

\title[Generalized Relativistic Langevin dynamics]{Asymptotic analysis for the Generalized Relativistic Langevin Equation}
\author{Ethan Baker$^1$, Manh Hong Duong$^1$ and Hung Dang Nguyen$^2$}
\address{$^1$ School of Mathematics, University of Birmingham, Birmingham, UK}

\address{$^2$ Department of Mathematics, University of Tennessee, Knoxville, Tennessee, USA}

\date{}

\begin{document}

\begin{abstract}
In this paper, we study a non-Markovian generalized relativistic Langevin equation (GRLE). We show that when the memory kernel is a sum of exponentials, the GRLE is equivalent to a Markovian system with added variables. We establish the well-posedness and polynomial ergodicity, obtaining an algebraic rate of convergence to the unique Gibbs distribution. From the Markovian GRLE, we recover the relativistic underdamped Langevin dynamics in a small-noise limit, as well as the classical (non-relativistic) generalized Langevin dynamics in the Newtonian limit. 
\end{abstract}

\maketitle

\section{Introduction}
\subsection{The generalized relativistic Langevin equation}
\noindent In this paper, we consider the following \textit{generalized relativistic Langevin equation (GRLE)}~\cite{chen2023persistent,petrosyan_relativistic_2022,zadra_translation-invariant_2023}:

\begin{subequations}
\label{GRLE}
 \begin{align}
 \label{qreletavistic}
\df q(t)&=\nabla K(p(t))\,\df t,\\
\df p(t)&=-\nabla U(q(t))\,\df t-\gamma\nabla K(p(t))\,\df t-\int_0^t \eta(t-s)\nabla K(p(s))\,\df s\,\df t\notag
\\&\qquad+\sqrt{2\gamma\beta^{-1}}\, \df W(t)+ F(t)\,\df t.\label{prelativistic}
\end{align}   
\end{subequations}
This stochastic differential equation describes the motion of a \textit{relativistic} particle with position $q$ and momentum $p$, both taking values in $\IR^d$. The first equation \eqref{qreletavistic} is simply a relation between the position and the momentum in the special relativistic setting where $K$ is the relativistic kinetic energy
\begin{equation} 
\label{relativistic KE}
K(p)=c\sqrt{m^2c^2+|p|^2}.
\end{equation}
Here, $c>0$ is the speed of light, and $m$ is the particle's mass at rest. The second equation 
\eqref{prelativistic} posits that the particle moves under the influence of different forces, namely, (i) an external force $-\nabla U$ where $U: \mathbb{R}^d\mapsto \mathbb{R}$ is an external potential; (ii) a friction $-\gamma \nabla K$ where $\gamma\geq 0$ is the friction coefficient; (iii) a delayed drag force from the fluid on the particle given by the integral $\int_0^t \eta(t-s)\nabla K(p(t))\,\df s\,\df t$ 
where $\eta:[0,\infty)\mapsto \IR^{d\times d}$ is a memory kernel; (iv) a random perturbation that consists of two components: a standard $d$-dimensional Brownian motion $(W(t))_{t\geq0}$ and a mean-zero stationary Gaussian process $F(t)$ satisfying the fluctuation–dissipation theorem \cite{kubo1966fluctuation, debbasch_relativistic_1997}:
\begin{equation}
\label{fluctuationdissipationtheorem}
    \mathbb{E}[F_i(t)F_j(s)]=\beta^{-1}(\eta(t-s))_{ij} \delta_{ij}, \qquad t\geq s,
\end{equation}
where $\delta_{ij}$ denotes the Kronecker delta. Finally, the parameter $\beta$ is the inverse temperature.
  
The system \eqref{qreletavistic}-\eqref{prelativistic} has been considered in \cite{chen2023persistent,petrosyan_relativistic_2022,zadra_translation-invariant_2023}\footnote{More precisely $\gamma=0$ in these papers.}---in particular in \cite{petrosyan_relativistic_2022,zadra_translation-invariant_2023} it is derived from a first-principle particle-bath Lagrangian. It is an extension of the classical (non-relativistic) generalized Langevin dynamics (GLE) \cite{ottobre_asymptotic_2011,pavliotis_stochastic_2014}, which is obtained from \eqref{qreletavistic}-\eqref{prelativistic} by replacing the relativistic kinetic energy \eqref{relativistic KE} by the classical one $K(p)=|p|^2/(2m)$,  to comply with Einstein’s theory of special relativity. Extending the classical Langevin dynamics to the relativistic systems is vital to many areas of physics including relativistic fluids/plasmas, effective field theories of dissipative hydrodynamics, and relativistic viscous electron flow
in graphene, just to name a few. The first relativistic Langevin model was introduced in \cite{debbasch_relativistic_1997}. Since then, it has become an active area of research, with a vast literature devoting to the study of relativistic systems. We refer the reader to \cite{debbasch2007relativistic, dunkel2009relativistic} for further information on various topics of the relativistic Langevin models. 

Analogous to the non-relativistic GLE, the system \eqref{qreletavistic}–\eqref{prelativistic} exhibits non-Markovian dynamics owing to the presence of the memory kernel. In conjunction with the nonlinearity induced by the relativistic kinetic energy, this renders the mathematical analysis considerably more intricate. It is known that, when the memory kernel is a sum of exponentials (also known as the Prony series memory kernel), the classical GLE can be equivalently formulated as a Markovian system by introducing additionally auxiliary variables \cite{kupferman_fractional_2004}. The Markovian system is more computationally tractable, and various efficient numerical schemes have been developed for its analysis \cite{baczewski2013numerical, duong_accurate_2022}.  Under this Markovian formulation, in \cite{ottobre_asymptotic_2011,duong2024asymptotic}, the authors prove the geometric ergodicity and obtain an exponential rate of convergence to the equilibrium,  which is an important problem in statistical physics, molecular dynamics and sampling techniques. In addition, \cite{ottobre_asymptotic_2011} also derives a white-noise limit of the Markovian GLE system, that is to show that, when the noise term is appropriately rescaled such that the correlation function becomes a Dirac measure, the GLE system converges to the underdamped Langevin dynamics, thus eliminating the auxiliary variables. 

Concerning relativistic Langevin dynamics, in addition to the questions of ergodicity and small mass or small noise limits that arise in classical models, another important issue is the Newtonian limit. Specifically, one seeks to show that as the speed of light tends to infinity, the relativistic system converges to the corresponding classical system. Establishing this limit ensures the consistency of the relativistic model with its classical counterpart in the non-relativistic regime. These topics for the system \eqref{GRLE} in the absence of the memory term have been studied recently by various authors \cite{alcantara2011relativistic,calogero2004newtonian, calogero2012exponential,arnold2025trend,duong_trend_2024,duong2026ergodicity}.
\subsection{Summary of the main results}
The aim of this paper is to extend the aforementioned works to the relativistic generalized Langevin dynamics \eqref{GRLE}. Under suitable assumptions on the potential $U$ and the kernel $\eta$, our main findings can be summarized as follows.
\begin{enumerate}
    \item (\textbf{Markovian formulation and well-posedness}) Proposition \ref{prop: Markovian formulation} shows that when the memory kernel {can be expressed as an exponential form}, cf. \eqref{markappmemorykernel}, the GRLE \eqref{GRLE} is equivalent to a Markovian system by introducing auxiliary variables. The well-posedness of this Markovian system is proved in Theorem \ref{existenceuniquenessmarkovtheorem}.
    \item (\textbf{Polynomial ergodicity}) Theorem \ref{polynomialerogicity} establishes a polynomial rate of convergence to the unique equilibrium measure, which is a Gibbs distribution, for the Markovian system.
    \item (\textbf{White-noise limit}) In Theorem \ref{relativisticwhitenoiselimit}, we show that under appropriate rescaling, the Markovian system is well approximated by the relativistic underdamped Langevin dynamics describing the evolution of the position and momentum variables, thus eliminating the auxiliary variables and obtaining the friction coefficient from the parameters of the Markovian GRLE system.
    \item (\textbf{Newtonian limit}) In Theorem \ref{relativisticnewtonianlimit}, we derive the Newtonian limit for the Markovian GRLE, recovering the Markovian GLE system when the parameter $c$, representing the speed of light, tends to infinity.
    \end{enumerate}    
The above results, as well as their relationship to existing results, are summarised in Figure \ref{fig:diagram}.
In order to precisely formulate the above results, we make the following assumption on the growth of the external potential.

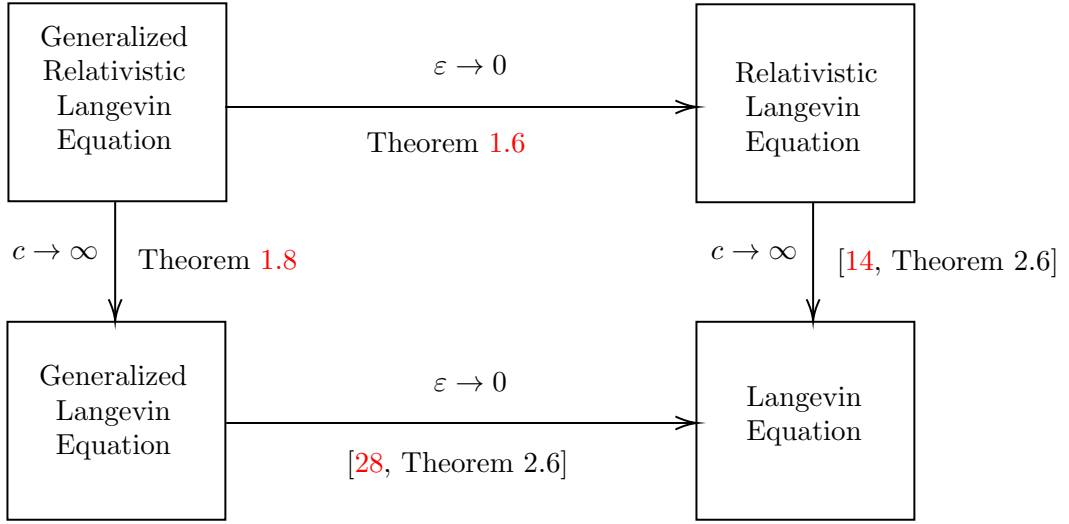
\begin{figure}[ht]
    \centering
    \input{diagram}
    \caption{Diagram of the relationships between Langevin equations.}
    \label{fig:diagram}
\end{figure}

\begin{assumption}
\label{UAssumptions}
The potential $U\in C^{\infty}\big(\IR^d;[1,\infty)\big)$ satisfies the following:
\label{potentialassumptions}
    \begin{enumerate}
        \item $\displaystyle{\langle\nabla U(q),q\rangle\geq a_1|q|^{\lambda+1}-a_2}$ for some $a_1,a_2>0$ and $\lambda\geq1$ for all $q\in\IR^d$; and,
        \item \label{Uqtest} $\displaystyle{\frac{1}{a_3}|q|^{\lambda+1}-a_3\leq U(q)\leq a_3\big(1+|q|^{\lambda+1}\big)}$ for some $a_3>0$.
    \end{enumerate}
\end{assumption}

\subsection{Markovian formulation}

We now describe the results in more detail. We start by introducing a Markovian Formulation for the GRLE. This is drawn upon the Markovian framework of the GLE as found in \cite{kupferman_fractional_2004,  leimkuhler_ergodic_2019, pavliotis_stochastic_2014}. More specifically, following \cite[Definition 1.3]{pavliotis_stochastic_2014}, two processes $X,Y$ are said to be equivalent if $X$ and $Y$ have the same finite dimensional distributions. Under the assumption that the memory kernel $\eta$ admits a special form of exponential functions, we assert that $(q,p)$ from \eqref{GRLE} is indeed equivalent to a Markov process solving a system of SDEs. This is rigorously verified through Proposition \ref{prop: Markovian formulation}, which can be regarded as a relativistic analogue of \cite[Proposition 8.1]{pavliotis_stochastic_2014}.

\begin{proposition}
\label{prop: Markovian formulation}
    Let $\Lambda\in\IR^{d\times k}$, and ${\bf{A}}\in\IR^{k\times k}$ be symmetric positive definite. Then, given the memory kernel can be written as,
 \begin{align}
 \label{markappmemorykernel}
        \eta(t)=\Lambda e^{-\bf{A}t}\Lambda^{\text{T}},
    \end{align}
$(q,p)$ given by \eqref{GRLE} is equivalent to the process $(q,p)$ solving,
\begin{subequations}
\label{markovapproximation}
    \begin{align}
        \df q(t)&=\nabla K(p(t))\, \df t,\\
        \df p(t)&=-\nabla U(q(t)) \, \df t -\gamma \nabla K(p(t))\df t+ \ \Lambda z(t)\, \df t+\sqrt{2\gamma\beta^{-1}}\,\df W(t), \label{markappp}\\
        \df z(t)&=-\Lambda^T \nabla K(p(t))\,\df t -{\bf{A}} z(t)\,\df t+\Sigma\, \df \W(t),
    \end{align}
\end{subequations}
for some $k$-dimensional Brownian motion $\W$, where $z:[0,\infty)\to \IR^{{k}}$ with the initial distribution given by,

\begin{equation*}
    z(0)\sim \mathcal{N}(0,\beta^{-1}{\bf{I}}_k),
\end{equation*} and the matrix $\Sigma\in \IR^{k\times k}$ satisfies $\Sigma\Sigma^{\text{T}}=2\beta^{-1}\bf{A}$.
\end{proposition}
The following theorem establishes the well-posedness of the Markovian system \eqref{markovapproximation}.
\begin{theorem}
\label{existenceuniquenessmarkovtheorem}
Under Assumption \ref{UAssumptions}, \eqref{markovapproximation} has a unique strong Markov solution.
\end{theorem}
The proof of this theorem is rather standard which relies on the Lipschitz property of $\nabla K(p)$ and a Lyapunov condition, see Section \ref{Markovian Approximation:Proof} for a further discussion of this point.

In order to study asypmtotic behaviors of the Markovian system \eqref{markovapproximation}, throughout the rest of the paper, we will consider a special case of \eqref{markovapproximation} that is widely explored in the literature \cite{duong2024asymptotic,  duong_accurate_2022, ottobre_asymptotic_2011}. More specifically,  we choose $k=Md$ for some $M\in\IN$, $\alpha=(\alpha_1,...,\alpha_M)\in(0,\infty)^M$ and $\lambda=(\lambda_1,...,\lambda_M)\in(0,\infty)^M$. We consider
\[\textbf{A}=\begin{pmatrix}
    \alpha_1{\textbf{I}}_d & \textbf{0} & \cdots & \textbf{0}\\
    {\bf{0}} & \alpha_2{\textbf{I}}_d & \cdots & {\bf{0}}\\
    \vdots & \vdots & \ddots & \vdots &\\
    \bf{0} & \bf{0} & \cdots & \alpha_M{\textbf{I}}_d,
\end{pmatrix}\]
and
\[\Lambda= {\begin{pmatrix}
    \lambda_1{\bf{I}}_d & \lambda_2{\bf{I}}_d &\cdots & \lambda_M{\bf{I}}_d
\end{pmatrix}}.\]
In view of \eqref{markappmemorykernel}, we restrict to the class of memory kernels (known as the Prony series memory kernel) given by
    \begin{equation}
\label{exponentialmemorykernel}
    \eta(t)=\sum_{i=1}^M \lambda_i^2 e^{-\alpha_i t}{\bf{I}}_d.
\end{equation}
   We then choose $z=(z_1,...,z_M)$ and $\W=(W_1,...,W_M)$, where $z_i\in\IR^d$ and $W_i$ are independent $d$-dimensional Brownian motions for $i=1,...,M$. Then, \eqref{markovapproximation} is reduced to
\begin{subequations}
\label{GRLEpqz}
 \begin{align}
 \label{qrelmarkov}
 \df q(t)&= \nabla K(p(t))\,\df t,\\
\df p(t)&=-\nabla U(q(t))\,\df t-\gamma \nabla K(p(t))\,\df t+\sum_{i=1}^M \lambda_i z_i(t)\,\df t+\sqrt{2\gamma\beta^{-1}}\, \df W(t), \label{prelmarkov}\\
\label{zrelmarkov}
\df z_i(t)&=-\lambda_i\ \nabla K(p(t))\,\df t-\alpha_i z_i(t)\,\df t+\sqrt{2\alpha_i\beta^{-1}}\,\df W_i(t),\quad i=1,\ldots, M,
\\q(0)&= q_0,\quad p(0)=p_0,\quad z_i(0)=z_{i,0}.\notag 
\end{align}
\end{subequations}

\subsection{Polynomial ergodicity} We turn to the asymptotic analysis of \eqref{GRLEpqz}, the first topic of which is the large-time stability. As a consequence of Theorem \ref{existenceuniquenessmarkovtheorem}, we may introduce the transition probabilities of $(q,p,z)$ solving \eqref{GRLEpqz} by
\begin{align*}
    P_t(x,B) = \IP\left[(q(t),p(t),z(t))\in B|(q(0),p(0),z(0))=x\right],
\end{align*}
which is well-defined for all $x\in \IR^{(2+M)d}$ and Borel sets $B\subset \IR^{(2+M)d}$. We recall that for a measurable space $(\Omega,\CF)$, the total variation norm for two probability measures $\mu$ and $\upsilon$ is given by,
    \[\left\|\mu-\upsilon\right\|_{\text{TV}}=\sup_{A\in\CF}\left|\mu(A)-\upsilon(A)\right|.\]  
We consider the following Gibbs distribution associated to \eqref{GRLEpqz}
\begin{equation}
    \label{Gibbs}
    \df\rho_{\beta}(q,p,z)=\frac{1}{Z}e^{-\beta H(q,p,z)} \,\df q  \df p \df z ,
    \end{equation}
where  the Hamiltonian $H$ is given by
\begin{equation}
\label{Hamiltonian} 
H\left(q,p,z\right)=U(q)+K(p)+\frac{1}{2}|z|^2, 
\end{equation}
and \begin{equation*}
        Z=\int_{\IR^{(2+M)d}}e^{-\beta H(q,p,z)}\df q\df p\df z.
\end{equation*}
denotes the normalisation constant.

The next result of the present paper characterizes the long-time behaviour of  \eqref{GRLEpqz} and obtains a polynomial rate of convergence to the unique invariant probability measure.
\begin{theorem}
\label{polynomialerogicity}
    Under Assumption \ref{UAssumptions}, let $(q,p,z)$ denote the solution to \eqref{GRLEpqz}. Suppose that either
    
    (i) $\gamma>0$; or
    
    (ii) $\gamma=0$ and $\nabla U$ is Lipschitz continuous.
    
Then the following hold:
    
    (1) for $c>0$ sufficiently large, the Gibbs distribution $\rho_\beta$ defined in \eqref{Gibbs} is the unique invariant probability measure for the process $(q,p,z)$; and 
    
    (2) for all $r\in\IN$, there exists $c>0$ sufficiently large and $V_r:\IR^{(2+M)d}\to[1,\infty)$ such that
   \begin{equation}
   \label{eq: ergodicity estimate}
    \left\|P_t(x,\cdot)- \rho_\beta\right\|_\text{TV}\leq \frac{C}{(1+t)^r}V_r(x),   \quad t\ge 0,
   \end{equation} 
    for some $C>0$ independent of $t$ and $x\in \IR^{(2+M)d}$.
\end{theorem}
To prove this theorem, we adopt the framework of \cite{hairer2009hot} (see \cite{ bakry2008rate, douc2009subgeometric, fort2005subgeometric} for earlier development), which consists of three main ingredients, namely, a H\"ormander's condition, a solvability of the associated control problem, and a suitable Lyapunov function quantifying the convergent rate. We will deal with these issues in Section \ref{sec: ergodicity}, and recall the framework of \cite{hairer2009hot} in more detail in Appendix \ref{Ergodicity Results}. 
\subsection{White-noise limit}
Now, we discuss the white-noise limit of the system \eqref{GRLEpqz}. Following the framework of \cite{nguyen2018small,ottobre_asymptotic_2011}, we introduce a diffusive rescaling to the memory kernel as follows
\begin{equation}
    \label{rescaledmemorykermel}
\eta^{(\varepsilon)}(t)=\sum_{i=1}^M \frac{\lambda_i^2}{\varepsilon} e^{-\frac{\alpha_i}{\varepsilon} t}\mathbf{I}_d,\quad \varepsilon\in(0,1).
\end{equation}
This amounts to rescaling $\lambda_i\mapsto \frac{\lambda_i}{\sqrt{\varepsilon}}$ and $\alpha_i\mapsto \frac{\alpha_i}{\varepsilon}$ in \eqref{GRLEpqz}.
In order to be consistent with \eqref{fluctuationdissipationtheorem}, the rescaled noise is given by
\begin{equation*}
\label{rescalednoise}
    F^{(\varepsilon)}(t):=\frac{1}{\sqrt{\varepsilon}}F\left(\frac{t}{\varepsilon}\right),
\end{equation*}
By applying this memory kernel to the Markovian formulation, \eqref{GRLEpqz} becomes
\begin{subequations}
\label{GRLEpqzrescaled}
 \begin{align}
\label{qrescaled} \df q^{(\varepsilon)}(t)&=\nabla K(p^{(\varepsilon)}(t))\,\df t,\\
 \notag \df p^{(\varepsilon)}(t)&=-\nabla U(q^{(\varepsilon)}(t))\,\df t-\gamma\nabla K(p^{(\varepsilon)}(t))\,\df t+\frac{1}{\sqrt{\varepsilon}}\sum_{i=1}^M \lambda_i z_i^{(\varepsilon)}(t)\,\df t\\ \label{prescaled}&\quad +\sqrt{2\gamma\beta^{-1}}\, \df W(t),\\
\label{zrescaled} \df z_i^{(\varepsilon)}(t)&=-\frac{\lambda_i}{\sqrt{\varepsilon}}\nabla K(p^{(\varepsilon)}(t))\,\df t-\frac{\alpha_i}{\varepsilon} z_i^{(\varepsilon)}(t)\,\df t+\sqrt{\frac{2\alpha_i}{{\varepsilon}}}\,\df W_i(t),\\
q^{(\varepsilon)}(0)&= q_0,\qquad p^{(\varepsilon)}(0)=p_0,\qquad z_i^{(\varepsilon)}(0)=z_{i,0},\notag 
\end{align}
\end{subequations}
with initial conditions independent of our choice of $\varepsilon$. Such a rescaling is commonly used, not only in the study of white noise limits \cite{ottobre_asymptotic_2011,nguyen2018small}, but also in the asymptotic analysis of stochastic partial differential equations, e.g., the stochastic Burgers' equation \cite{cannizzaro_gaussian_2024}. Note that we have used the superscripts in \eqref{GRLEpqzrescaled} to indicate that the dynamics depends on the parameters $\varepsilon$ which will be important in the study of the white-noise in the next theorem.

As usual, we require the following assumption on the initial conditions:
\begin{assumption}
\label{boundedinitialconditionsrel}
    The initial conditions $(q_0,p_0,z_0)\in \mathbb{R}^{(2+M)d}$ have finite moments: for all $n\in\IN$,
    \begin{equation*}
\IE\Big[\left|p_0\right|^n+\left|q_0\right|^n+\sum_{i=1}^M\left|z_{i,0}\right|^n\Big]<\infty.
\end{equation*}
\end{assumption}
We now state the third main result of the paper giving the white noise limit for the GRLE.
\begin{theorem}[The Relativistic White Noise Limit]
    \label{relativisticwhitenoiselimit}
Under Assumption \ref{UAssumptions}, let $\left(q^{(\varepsilon)},p^{(\varepsilon)},z^{(\varepsilon)}\right)$ be the solution to \eqref{GRLEpqzrescaled} with initial conditions $(q_0, p_0, z_0)$ satisfying Assumption \ref{boundedinitialconditionsrel}. For all $T>0$ and $n\in \IN$, it holds that
\begin{align}
\IE\Big[\sup_{t\in[0,T]}|q^{(\varepsilon)}(t)-Q(t)|^n+ \sup_{t\in[0,T]}|p^{(\varepsilon)}(t)-P(t)|^n\Big]\to 0
\end{align}
as $\varepsilon\to 0$, where $(Q,P)$ satisfy the following underdamped relativistic Langevin dynamics
\begin{subequations}
\label{PQprocess}
 \begin{align}
\df Q(t)&=\nabla K(P(t))\,\df t,\\
\label{bigP} \df P(t)&=-\nabla U(Q(t))\,\df t-\Big( \gamma+\sum_{i=1}^M\frac{\lambda^2_i}{\alpha_i} \Big) \nabla K(P(t))\,\df t+\sqrt{2\beta^{-1}   \gamma} \, d W(t)+ \sum_{i=1}^M\sqrt{  \frac{2\beta^{-1} \lambda_i^2}{\alpha_i}} \, \df W_i(t),\\
Q(0)&=q_0,\quad P(0)=p_0.
\end{align}   
\end{subequations}
\end{theorem}
The relativistic Langevin dynamics \eqref{PQprocess} was introduce in \cite{debbasch_relativistic_1997} and has been studied intensively in the literature, see for instance the recent paper \cite{duong_trend_2024} and references therein. The proof of Theorem \ref{relativisticwhitenoiselimit} relies on the fact that $z_i^{(\varepsilon)}$ is an Ornstein-Uhlenbeck process, which can be solved explicitly. Then by substituting it to \eqref{prescaled}, we obtain a closed system for the evolution of position and momentum. This will be rigorously established in Section \ref{White Noise and Newtonian Limits}. 
\begin{remark}
As a consequence of Theorem \ref{relativisticwhitenoiselimit}, as $\varepsilon\to 0$, the process $(q^{(\varepsilon)},p^{(\varepsilon)})$ converges in distribution to the solution of the following more familiar-looking RLE \cite{debbasch_relativistic_1997}
\begin{subequations}
\label{RLE}
\begin{align}
\df Q(t) &= \nabla K(P(t))\,\df t,\\
\df P(t) & = -\nabla U(Q(t))\, \df t -\gamma^{*}  \nabla K(P(t)) \, \df t+\sqrt{2\gamma^{*}\beta^{-1}} \, \df W(t),\\
Q(0)&=q_0,\qquad P(0)=p_0,\notag 
\end{align}
\end{subequations}
where the friction coefficient $\gamma^*$ is given by
\[\gamma^{*}=\gamma+\sum_{i=1}^M\frac{\lambda_i^{2}}{\alpha_i}.\]
Note that, comparing to \eqref{PQprocess}, the above system consists of only one Wiener process and one effective friction coefficient $\gamma^*$.
\end{remark}
\subsection{Newtonian limit} 
We turn our attention to the Newtonian limit of system\eqref{GRLEpqz} when the speed of light $c$ tends to infinity. Our main finding is that, in this regime, one may recover the Markovian formulation of the non-relativistic generalised Langevin dynamics from \eqref{GRLEpqz}. This is summarized below through Theorem \ref{relativisticnewtonianlimit}.
\begin{theorem}[The Newtonian Limit for the GRLE]
\label{relativisticnewtonianlimit}
Under Assumption \ref{UAssumptions}, let $(q^{(c)},p^{(c)},z^{(c)})$ be the solution to \eqref{GRLEpqz} with initial conditions $(q_0,p_0,z_0)$ satisfying Assumption \ref{boundedinitialconditionsrel}. Then, for all $T>0$ and $n>0$, it holds that
\begin{align} \label{lim:newtonian}
    \IE\Big[\sup_{t\in[0,T]}\left|q^{(c)}(t)-q(t)\right|^n+\sup_{t\in[0,T]}\left|p^{(c)}(t)-p(t)\right|^n+\sum_{i=1}^M\sup_{t\in[0,T]}\left|z_i^{(c)}(t)-z_i(t)\right|^n\Big]\to 0,
\end{align}
as $c\to\infty$, where the process $(q,p,z)$ satisfy the classical (non-relativistic) generalized Langevin equation \cite{ottobre_asymptotic_2011}
\begin{subequations}
\label{GLEpqz}
 \begin{align}
 \label{qclasmarkov}
 \df q(t)&=\frac{p(t)}{m}\,\df t,\\
 \label{pclasmarkov}
\df p(t)&=-\nabla U(q(t))\,\df t-\gamma\frac{p(t)}{m}\,\df t+\sum_{i=1}^M \lambda_i z_i(t)\,\df t+\sqrt{2\gamma\beta^{-1}}\, \df W(t),\\
\label{zclasmarkov}
\df z_i(t)&=-\lambda_i\frac{p(t)}{m}\,\df t-\alpha_i z_i(t)\,\df t+\sqrt{2\alpha_i\beta^{-1}}\,\df W_i(t),\\
q(0)&=q_0,\qquad p(0)=p_0,\qquad z_i(0)=z_{i,0}. \notag 
\end{align}   
\end{subequations}
\end{theorem}
The proof of this theorem is based on estimating the errors between the two systems \eqref{GRLEpqz} and \eqref{GLEpqz} while making use of Gronwall's lemma and suitable moment bounds. All of this will be addressed in detail in Section \ref{White Noise and Newtonian Limits}.

\subsection{Future works} In this paper, we focus on the most popular type of memory kernel, namely the Prony series memory kernel as a finite sum of exponentials. For this class, the non-Markovian GRLE can be equivalently written as a Markovian one by augmenting the original process by a finite number of auxiliary variables $z$. Other class of memory kernels have also been considered in the literature, notably the power-law ones \cite{glatt2020generalized,nguyen2018small}. In this latter case, one needs to use an infinite number of auxiliary variables and study infinite-dimensional stochastic differential equations in Hilbert spaces. It would be an interesting problem to extend the results of the present paper to such settings. Another challenging direction for future research is to investigate systems of interacting particles, in particular when the interaction potentials are singular such as the Coloumb and Lennard-Jones potentials. These interacting particle systems play a central role in statistical physics and have received considerable attention in recent years \cite{duong2024asymptotic,duong_trend_2024,duong2026ergodicity,jabin2017mean,serfaty2018systems}.  
\subsection{Organization of the paper}
The rest of the paper is structured as follows. In Section \ref{Markovian Approximation:Proof}, we derive the Markovian formulation of the GRLE and establish the well-posedness of the resulting Markov system \eqref{markovapproximation}. Section \ref{sec: ergodicity} studies the ergodicity of the Markovian system when the memory kernel $\eta$ admits the form as a finite sum of exponentials \eqref{exponentialmemorykernel}. Particularly, in Section \ref{Lyapunov Functions for the Generalised Relativistic Langevin Equation}, we construct Lyapunov functions for the GRLE whereas in Section \ref{Polynomial Ergodicity}, we complete our proof of polynomial ergodicity by showing that Hörmander's condition holds and the control problem associated with the Markovian formulation is solvable. In Section \ref{White Noise and Newtonian Limits}, we prove the white noise and Newtonian limits of the Markovian formulation; first when $\nabla U$ is Lipschitz, and then extending to the class of potentials satisfying Assumption \ref{UAssumptions}. The paper concludes with two appendices. In Appendix \ref{Ergodicity Results}, we recall the framework developed in \cite{hairer2009hot} that we use in order to prove the polynomial ergodicity of the GRLE, while in Appendix \ref{appendix2}, we present a technical lemma that is used in the proofs of the main results.

\section{Well-posdeness of the Markovian GRLE}
\label{Markovian Approximation:Proof}
In this section, we consider the Markovian system \eqref{markovapproximation} and establish Proposition \ref{prop: Markovian formulation} showing that \eqref{markovapproximation} is indeed equivalent to the non-Markovian GRLE \eqref{GRLE} when the memory kernel is given by the exponential form \eqref{markappmemorykernel}. We also present the proof of Theorem \ref{existenceuniquenessmarkovtheorem} deducing the well-posedness of \eqref{markovapproximation}.

We start with the proof of Proposition \ref{prop: Markovian formulation}, whose argument mainly relies on Duhamel's formulas.

\begin{proof}[Proof of Proposition \ref{prop: Markovian formulation}]
        By the Duhamel's formula, we note that the process $z$ in \eqref{markovapproximation} can be recast as
    \begin{equation}
    \label{markappzito}
       z(t)=e^{-{\bf{A}}t}z(0)+\int_0^te^{-A(t-s)}\Sigma\, \df W(s) -\int_0^te^{-{\bf{A}}(t-s)}\Lambda^T\nabla K(p(s))\,\df s.
    \end{equation}
Substituting (\ref{markappzito}) into (\ref{markappp}), we get
\begin{align*}
   \df p(t)&=-\nabla U(q(t))\,\df t-\gamma\nabla K(p(t))\,\df t-\int_0^t \eta(t-s)\nabla K(p(t))\,\df s\,\df t\\
   &\qquad+\sqrt{2\gamma\beta^{-1}}\, \df W(t)+ F(t)\,\df t,
\end{align*}
where $\eta$ is defined as in (\ref{markappmemorykernel}), and $(F(t))_{t\geq0}$ is an Ornstein-Uhlenbeck process given by
\[
F(t)=\Lambda e^{-{\bf{A}}t}z(0)+\Lambda\int_0^te^{-{\bf{A}}(t-s)}\Sigma \,\df W(s).
\]
By using a routine calculation similar to the approach found in \cite[Proposition 8.1]{pavliotis_stochastic_2014}, $F$ satisfies \eqref{fluctuationdissipationtheorem}. This completes the proof.
\end{proof}

Next, we turn to the well-posedness of \eqref{markovapproximation} and establish Theorem \ref{existenceuniquenessmarkovtheorem}. In order to do this, we need the following lemma showing that $\nabla K$ is globally Lipschitz continuous.
\begin{lemma}
\label{relativistickineticenergylipschitzcontinuous}
    The function $K$ defined in \eqref{relativistic KE} satisfies that $\nabla  K$ is Lipschitz continuous and that $\nabla^2  K$ is bounded. In particular, for all $p_1,p_2\in\IR^d$,
    \begin{align*}
        |\nabla K(p_1)-\nabla K(p_2)|\leq\frac{d}{m}|p_1-p_2|,
    \end{align*}
\end{lemma}

\begin{proof}
Observe that, for $1\leq i,j\leq M$ where $i\neq j$,
    \begin{align}
    \label{derivativep_i2}
\left|\frac{\delta^2}{\delta p_i^2} K(p)\right|=\frac{\left|c^3m^2+c|p|^2-cp_i^2\right|}{\left(|p|^2+c^2m^2 \right)^\frac{3}{2}}\leq \frac{1}{m},
\end{align}
and,
    \begin{align}
    \label{derivativep_ip_j}
\left|\frac{\delta^2}{\delta p_i\delta p_j} K(p)\right|=\frac{c|p_ip_j|}{\left(|p|^2+c^2m^2 \right)^\frac{3}{2}}\leq \frac{1}{m},
\end{align}
Together \eqref{derivativep_i2} and \eqref{derivativep_ip_j} imply that 
\begin{equation*}
    \left\|\nabla^2 K \right\|_\text{F}\leq \frac{d}{m},
\end{equation*} 
where $\nabla ^2K$ denotes the Hessian of $K$ and $\left\|\cdot\right\|_\text{F}$ denotes the Frobenius norm.

\end{proof}

To prove Theorem \ref{existenceuniquenessmarkovtheorem}, we will apply \cite[Theorem 3.5]{khasminskii_stochastic_2012} which states that there exists a unique Markov solution to a time homogeneous SDE,
\[
\df X(t)=a(X(t))\,\df t+\sigma(X(t))\, \df W(t),
\]
if, for all $R>0$, there exists $C_R>0$ such that, whenever $|x|,|y|<R$,

\begin{equation}
\label{defforlocallipschitzcont}
    |a(y)-a(x)|+\|\sigma(y)-\sigma(x)\|\leq C_R|y-x|;
\end{equation}
and, there exists $V\in C^2(\IR)$ such that $\lim_{|x|\to\infty}V(x)=\infty$, and

\begin{equation}
\label{igexistencecondition}
    \CL V\leq \zeta V,
\end{equation}
for some $\zeta>0$, where $\CL$ is the generator of the SDE. See also \cite{veretennikov2024lyapunov}.

\begin{proof}[Proof of Theorem \ref{existenceuniquenessmarkovtheorem}]
We proceed to verify conditions \eqref{defforlocallipschitzcont} and \eqref{igexistencecondition} for the Markovian system \eqref{markovapproximation} whose generator is given by
\begin{align*}
    \CL f(q,p,z)& =\nabla K(p)\cdot \nabla_q f-(\nabla U(q)+\gamma\nabla_p K(p)-\Lambda z)\cdot \nabla_p f-(\Lambda^T \nabla K(p)+\mathbf{A}z)\cdot\nabla_z f\\
    &\qquad+\gamma\beta^{-1}\Delta_p f+\frac{1}{2}\mathrm{Tr}\Big(\Sigma\Sigma^T\nabla_z^2 f\Big).
\end{align*}
Condition \eqref{defforlocallipschitzcont} follows from Lemma \ref{relativistickineticenergylipschitzcontinuous} and the fact that $U\in C^{\infty}(\IR)$. Concerning condition \eqref{igexistencecondition}, we recall the formula of the Hamiltonian of the system \eqref{markovapproximation} defined in \eqref{Hamiltonian}:
\begin{equation*}
H\left(q,p,z\right)=U(q)+K(p)+\frac{1}{2}|z|^2.   
\end{equation*}
Applying $\CL$ to $H$, we have
\begin{align*}
   \CL H&= -\gamma c^2\frac{|p|^2}{{m^2c^2+|p|^2}}-z^T {\bf{A}}z+\gamma\beta^{-1}\Delta K (p)+\frac{1}{2}\text{tr}\Sigma \Sigma^T
   \\&\leq C,
\end{align*}
for some $C>0$. The last implication follows from the hypothesis that $\bf{A}$ is positive definite and Lemma \ref{relativistickineticenergylipschitzcontinuous}. It follows that condition \eqref{igexistencecondition} is satisfied by choosing $ V=H+1$.
\end{proof}

\section{Ergodicity}
\label{sec: ergodicity}
In this section, we prove Theorem \ref{polynomialerogicity} on the polynomial ergodicity of the Markovian system \eqref{GRLEpqz} when $\eta$ can be expressed as a finite sum of exponentials as in \eqref{exponentialmemorykernel}. As already mentioned in the introduction, we adopt the framework of \cite{hairer2009hot}, which in turn was built upon the techniques developed earlier in \cite{ bakry2008rate, douc2009subgeometric, fort2005subgeometric}. For self-consistency, the framework of \cite{hairer2009hot} will be summarized in Appendix \ref{Ergodicity Results}. Following this framework, the proof of estimate \eqref{eq: ergodicity estimate} consists of three main steps, namely, verifying the H\"ormander's condition, proving the solvability of the associated control problem, and constructing a suitable Lyapunov function.

The H\"ormander's condition, which ensures the smoothness of the system's transition probability, will be verified by applying the classical H\"ormander's Theorem \cite{hormander1967hypoelliptic}, which asserts that the state space may be generated by the collection of vector fields jointly induced by the diffusion and the drifts. Since we are dealing with finite-dimensional settings, this will follow from direct computations on Lie brackets. The solvability of the associated control problem can be established by the Support Theorem \cite{stroock1972degenerate} showing that one can always find appropriate controls allowing for driving the dynamics to any bounded ball. The proof of H\"ormander theorem and solvability condition are rather standard and will be discussed in Section \ref{Polynomial Ergodicity}.
The third step requires the construction of a Lyapunov function, which is an energy-like function $V$ satisfying an inequality of the form
\begin{align} \label{ineq:Lyapunov:polynomial-mixing}
\frac{d}{d t}\mathbb{E}\big[ V(q(t),p(t),z(t)) \big]\leq  -c_1\mathbb{E}\big[ V(q(t),p(t),z(t))^\alpha\big]+c_2,\quad t\geq 0,
\end{align}
for a suitable constant $\alpha\in(0,1]$. The construction of such a function $V$ is highly nontrivial due to the lack of strong dissipation in the $p$-direction as well as the impact of the nonlinearity of the relativistic kinetic energy. We are only able to prove \eqref{ineq:Lyapunov:polynomial-mixing} for some $\alpha\in(0,1)$, thus yielding only an algebraic mixing rate as stated in Theorem \ref{polynomialerogicity}. This will be done in Section \ref{Lyapunov Functions for the Generalised Relativistic Langevin Equation}.
\subsection{Invariant measure}
We first prove that  the Gibbs distribution given by \eqref{Gibbs} is indeed an invariant probability measure for \eqref{markovapproximation} (and thus for \eqref{GRLEpqz}).
\begin{lemma}
Suppose $U\in C^{\infty}(\IR)$ satisfies Assumption \ref{UAssumptions}. Then the Gibbs distribution $\rho_\beta$ defined in \eqref{Gibbs} is an invariant measure of \eqref{markovapproximation}.
\end{lemma}
\begin{proof}
Note that \eqref{markovapproximation} can be written in a more compact form as follows
\begin{equation}
\df X(t)=J\nabla H(X(t))\,\df t-D\nabla H(X(t))\,\df t+\sqrt{2\beta^{-1}D}\,\df\mathbf{W}(t),    
\label{eq: X}
\end{equation}
where $X=(q,p,z)^T$ and
\begin{equation}
\label{J and D}
J=\begin{pmatrix}
    0&1&0\\
    -1&0&\Lambda\\
    0&-\Lambda^T&0
\end{pmatrix},\quad D=\begin{pmatrix}
    0&0&0\\
    0&\gamma&0\\
    0&0&\mathbf{A}
\end{pmatrix},\quad \mathbf{W}(t)=\begin{pmatrix}
    0\\ W(t)\\ \W(t)
\end{pmatrix}.    
\end{equation}
It is clear that $J$ is anti-symmetric and $D$ is symmetric positive semi-definite.
In fact, to show that  \eqref{eq: X} is indeed the same as \eqref{markovapproximation}, we compute each term in the RHS of \eqref{eq: X} explicitly:
\begin{align*}
&J\nabla H=\begin{pmatrix}
    0&1&0\\
    -1&0&\Lambda\\
    0&-\Lambda^T&0
\end{pmatrix}\begin{pmatrix}
    \nabla U(q)\\ \nabla K(p)\\ z
\end{pmatrix}=\begin{pmatrix}
    \nabla K(p)\\
    -\nabla U(q)+\Lambda z\\
    -\Lambda^T \nabla K(p)
\end{pmatrix},
\\& D\nabla H=D=\begin{pmatrix}
    0&0&0\\
    0&\gamma&0\\
    0&0&\mathbf{A}
\end{pmatrix}\begin{pmatrix}
    \nabla U(q)\\ \nabla K(p)\\ z
\end{pmatrix}=\begin{pmatrix}
0\\ \gamma\nabla K(p)\\ \mathbf{A}z
\end{pmatrix},
\\& \sqrt{2\beta^{-1}D}=\begin{pmatrix}
    0&0&0\\
    0&\sqrt{2\beta^{-1}\gamma}&0\\
    0&0&\sqrt{2\beta^{-1}\mathbf{A}}
\end{pmatrix}=\begin{pmatrix}
    0&0&0\\
    0&\sqrt{2\beta^{-1}\gamma}&0\\
    0&0&\Sigma
    \end{pmatrix}.
\end{align*}
Substituting these expressions back to \eqref{eq: X} we get
\begin{align*}
\df \begin{pmatrix}
 q(t)\\
 p(t)\\
 z(t)
\end{pmatrix}&=  \begin{pmatrix}
    \nabla K(p)\\
    -\nabla U(q)+\Lambda z\\
    -\Lambda^T \nabla K(p)
\end{pmatrix}\, \df t-\begin{pmatrix}
0\\ \gamma\nabla K(p)\\ \mathbf{A}z
\end{pmatrix}\,\df t+\begin{pmatrix}
    0&0&0\\
    0&\sqrt{2\beta^{-1}\gamma}&0\\
    0&0&\Sigma
    \end{pmatrix}\begin{pmatrix}
        0\\ \df W(t)\\ \df \W(t)
    \end{pmatrix}
\\&=\begin{pmatrix}
    \nabla K(p)\\
    -\nabla U(p)+\Lambda z-\gamma\nabla K(p)\\
    -\Lambda^T \nabla K(p)-\mathbf{A}z
\end{pmatrix}\,\df t+\begin{pmatrix}
    0\\ \sqrt{2\beta^{-1}\gamma}\,\df W(t)\\ \Sigma \df \W(t)
\end{pmatrix},
\end{align*}
which is precisely \eqref{markovapproximation}. The advantage of the compact form \eqref{eq: X} is that it is very convenient for the verification that $\rho_\beta$ is an invariant measure of \eqref{markovapproximation}. Indeed, using this form, the adjoint generator of \eqref{markovapproximation}, $\CL^*$, is given by
\[
\CL^*\rho=-\mathrm{div}(J\nabla H\rho)+\mathrm{div}(D\nabla H\rho)+\beta^{-1}\mathrm{div}(D\nabla\rho)=-\mathrm{div}(J\nabla H\rho)+\mathrm{div}[D(\nabla H\rho+\beta^{-1}\nabla\rho)].
\]
When $\rho=\rho_\beta$ we have $
\nabla\rho_\beta=-\beta\nabla H\rho$. In addition  we have $\mathrm{div}(J\nabla H)=J\nabla H\cdot\nabla H=0$ due to the anti-symmetry of $J$. It follows that $\beta^{-1}\nabla\rho+\nabla H\rho=0$ and
\begin{align*}
&\mathrm{div}(J\nabla H\rho)=\mathrm{div}(J\nabla H)\rho+J\nabla H\cdot\nabla \rho=\mathrm{div}(J\nabla H)\rho-\beta J\nabla H\cdot\nabla H \rho=0-0=0.
\end{align*}
Therefore, $\CL^*\rho_\beta=0$, implying $\rho_\beta$ is an invariant measure of \eqref{markovapproximation}, as claimed.
\end{proof}

\subsection{Construction of Lyapunov functions}
\label{Lyapunov Functions for the Generalised Relativistic Langevin Equation}
Since the mass constant $m$ does not affect the analysis, throughout the rest of the paper, we assume $m=1$. We also adopt the notation
\[\epsilon=\frac{1}{c^2}.\]
So, the relativistic kinetic energy $K(p)$ from \eqref{relativistic KE} is reduced to
\begin{align*}
    K(p) = \frac{1}{\epsilon}\sqrt{1+\epsilon|p|^2}.
\end{align*}
Inspired by \cite{duong2024asymptotic}, we now construct $\phi$-Lyapunov functions (see Definition \ref{philyapunov}) for the GRLE, where $\phi (t)=\zeta t^{1-\frac{1}{2n}}$ for some positive integer $n$ and $\zeta>0$. We first consider the simpler case when $\gamma>0$.
\begin{proposition}
\label{lyaponovconditiongammageq0}
    Let $V_1$ be defined by,
    \begin{equation} \label{form:V_1}
        V_1(q,p,z)=H_1(q,p,z)^2+\epsilon\langle q,p\rangle+\kappa,
    \end{equation}
where,
\begin{align} \label{hamiltonianforlyapunovfunction}
H_1(q,p,z)=\epsilon H(q,p,z) = \epsilon U(q) + \sqrt{1+\epsilon|p|^2}+\frac{1}{2}\epsilon\|z\|^2,
\end{align}
and $\kappa$ is sufficiently large. Then for all $n\in\IN$ and $\gamma>0$, 
\begin{equation}
\label{lyapunovcondition}
    \CL V_1^n\leq -\zeta V_1^{n-\frac{1}{2}}+C,
\end{equation}
for all sufficiently small $\epsilon=\epsilon(n,\gamma,\alpha, \lambda,M)>0$ and for some positive constants $\zeta=\zeta(n,\gamma,\alpha, \lambda, M)$ and $C=C(\epsilon,n,\gamma, \alpha,\lambda,M,\epsilon)$. In the above, $\CL$ denotes the infinitesimal operator for \eqref{GRLEpqz}, which is given by
\begin{equation}
\CL f(q,p,z)=J\nabla H\cdot \nabla f -D\nabla H\cdot \nabla f +\beta^{-1}\mathrm{div}(D\nabla f),
\end{equation}
where $J$ and $D$ are given in \eqref{J and D} with $\Lambda=\mathrm{diag}(\lambda_1,\ldots, \lambda_M)$ and $\bf{A}=\mathrm{diag}(\alpha_1,\ldots,\alpha_M)$.
\end{proposition}
The form of the Lyapunov function given by \eqref{form:V_1} is motivated from the Lyapunov approach in \cite{duong2024asymptotic}. We notice that, instead of the Hamiltonian as often used for the classical Langevin dynamics, we need to employ its square due to the weak dissipation from the relativistic kinetic energy $K(p)$. On the other hand, the idea of perturbing a function of Hamiltonian by lower order terms such as the cross term $\langle q,p\rangle$ in \eqref{form:V_1} is well-known in the literature in the framework of  hypocoercivity method \cite{villani2009hypocoercivity}. It is a powerful method and has been employed widely when constructing Lyapunov functions for degenerate systems where the noises act only on some of direction of the phase spaces, see for instance \cite{duong2024asymptotic, hairer2009hot,  lu2019geometric} for Langevin-type dynamics. 

\begin{proof}[Proof of Proposition \ref{lyaponovconditiongammageq0}]
In the below, $C,\zeta_1>0$ may change from line to line, and $\zeta_1$ is independent of $\epsilon$. We first consider the case when $n=1$ and apply the generator $\mathcal{L}$ to $H_1$ to obtain
\begin{align}
      \mathcal{L}H_1&=-\gamma\frac{\epsilon|p|^2}{1+\epsilon|p|^2}+\gamma\beta^{-1}\frac{d\epsilon+(d-1)\epsilon^2|p|^2}{(1+\epsilon|p|^2)^\frac{3}{2}}-\epsilon\sum_{i=1}^M\alpha_i| z_i|^2+\epsilon d\sum_{i=1}^M\alpha_i\beta^{-1}\label{infinitesimaloperatorH}\\
    &\leq -\gamma\frac{\epsilon|p|^2}{1+\epsilon|p|^2}+\gamma\beta^{-1}(2d-1)\epsilon-\epsilon\sum_{i=1}^M\alpha_i| z_i|^2+\epsilon d\sum_{i=1}^M\alpha_i\beta^{-1}. \notag
\end{align}
As a consequence, applying $\mathcal{L}$ to $H_1^2$, we get
\begin{align}
\mathcal{L}H_1^2=&2H_1\mathcal{L}H_1+2\gamma\beta^{-1}\frac{\epsilon^2|p|^2}{1+\epsilon|p|^2}+\epsilon^2\sum_{k=1}^M2\alpha_i\beta^{-1}|z_i|^2 \label{infinitesimaloperatorH2}\\
    \leq &2H_1\CL H_1+2\gamma\beta^{-1}\epsilon+\epsilon^2\sum_{i=1}^M2\alpha_k\beta^{-1}|z_i|^2. \notag
\end{align}
Next, considering the cross term $\langle q,p\rangle$, we have

\[\mathcal{L}\langle q,p\rangle=\langle p,\nabla K(p) \rangle -\langle q,\nabla U(q) \rangle -\gamma \langle q,\nabla K(p) \rangle +\sum_{i=1}^M\lambda_i \langle q,z_i \rangle .\]
By applying Assumption \ref{UAssumptions} and Cauchy-Schwarz inequality,

\begin{align}
    \CL\langle q,p \rangle &\leq \frac{|p|^2}{\sqrt{1+\epsilon|p|^2}}-\zeta_1|q|^{\lambda+1}+C-\gamma\frac{\langle q,p\rangle}{\sqrt{1+\epsilon|p|^2}}+\sum_{i=1}^M\lambda_i\langle q,z_i \rangle\\
    &\leq \frac{|p|+\gamma|q|}{\sqrt{\epsilon}}-\zeta_1|q|^{\lambda+1}+C+\sum_{i=1}^M\lambda_i\langle q,z_i \rangle.
\end{align}
So, we have
\begin{align*}
    \CL\epsilon\left\langle q,p\right\rangle &\leq \sqrt{\epsilon}|p|+\sqrt{\epsilon}\gamma^2|q|+C-\epsilon \zeta_1|q|^{\lambda+1}+\sum_{i=1}^M\lambda_i|q|^2+\epsilon^2\sum_{i=1}^M\lambda_i|z_i|^2,\\
    &\leq \sqrt{\epsilon}|p| -\frac{\epsilon}{2}\zeta_1|q|^{\lambda+1}+\epsilon^2\sum_{i=1}^M\lambda_1|z_i|^2+C.
\end{align*}
Putting the above together, we obtain

\begin{align}
    \mathcal{L}\left(H_1^2+ \epsilon\langle q,p\rangle\right) &\leq 2H_1\left(-\gamma\frac{\epsilon|p|^2}{1+\epsilon|p|^2}+(2d+1)\epsilon-\epsilon\sum_{i=1}^M\alpha_i| z_i|^2+\epsilon d\sum_{i=1}^M\alpha_i\beta\right)\notag \\&+\sqrt{\epsilon}|p| -\frac{\epsilon}{2}\zeta_1|q|^{\lambda+1}+\epsilon^2\sum_{i=1}^M\lambda_1|z_i|^2+C+2\gamma\beta^{-1}\epsilon+\epsilon^2\sum_{i=1}^M2\alpha_i\beta^{-1}|z_i|^2. \label{firstV1inequality}
\end{align}
We have the following inequality,
\begin{align}
\label{innerproductpkpbound}
    \epsilon\langle p,\nabla K(p)\rangle= \frac{\epsilon|p|^2}{\sqrt{1+\epsilon|p|^2}}\geq \sqrt{1+\epsilon|p|^2}-1.
\end{align}
Using \eqref{innerproductpkpbound}, as well as the inequality
\begin{equation}
\label{H1inequality}
    H_1\geq \sqrt{1+\epsilon|p|^2}+\frac{\epsilon}{2}\sum_{i=1}^M|z_i|^2\ge \frac{\sqrt{\epsilon}|p|}{2}+\frac{1}{2}+\frac{\epsilon}{2}\sum_{i=1}^M|z_i|^2,
\end{equation}
we find
\begin{align}
     &2H_1\left(-\gamma\frac{\epsilon|p|^2}{1+\epsilon|p|^2}+(2d+1)\epsilon-\epsilon\sum_{i=1}^M\alpha_i| z_i|^2+\epsilon d\sum_{i=1}^M\alpha_i\beta\right) \notag \\&\leq -2\gamma\frac{\epsilon|p|^2}{\sqrt{1+\epsilon|p|^2}}+2(2d-1)\epsilon H_1-2\epsilon H_1\sum_{i=1}^M\alpha_i| z_i|^2 + 2\epsilon d\sum_{i=1}^M\alpha_i\beta H_1, \notag\\
     & \label{LHinequality} \leq -2\gamma\sqrt{1+\epsilon|p|^2}+2(2d-1)\epsilon H_1+2\epsilon d\sum_{i=1}^M\alpha_i\beta H_1-\epsilon \sum_{i=1}^M\alpha_i| z_i|^2-\epsilon^2\sum_{i=1}^M\alpha_i|z_i|^4.
\end{align}
By combining \eqref{firstV1inequality} and \eqref{LHinequality}, we get
\begin{align*}
    \mathcal{L}\left(H_1^2+\epsilon_1\langle q,p\rangle\right)\leq-&\gamma\sqrt{1+\epsilon|p|^2}+2(2d-1)\epsilon H_1+2\epsilon d\sum_{i=1}^M\alpha_i\beta H_1\\&-\frac{\epsilon}{2}\zeta_1|q|^{\lambda+1}-\sum_{i=1}^M\alpha_i\epsilon|z_i|^2+\epsilon^2\sum_{i=1}^M\left((\lambda_i+\alpha\beta^{-1})|z_i|^2-\alpha_i|z_i|^4\right)+C.
\end{align*}
By applying Assumption \ref{UAssumptions}, it holds that
\begin{align*}
    -\frac{\epsilon}{2}|q|^{\lambda+1}\leq -\epsilon \zeta_1 U(q)+C.
\end{align*}
Thus, choosing $\epsilon$ sufficiently small, we can estimate
\begin{align}
    \CL V_1 &\leq-\zeta_1 H_1+C, \label{Lbound} 
\end{align}
By our definition of $V_1$, we can choose $\epsilon>0$ and $\kappa>0$ such that,
\begin{equation}
\label{Vbound}
    \zeta_1 H_1^2-C\leq V_1\leq CH_1^2+C.
\end{equation}
Note that $V_1$ is also bounded below by 1. Using this fact, and combining \eqref{Lbound} and \eqref{Vbound},
\begin{align}
\label{Lvbound}
    \mathcal{L}V_1\leq-\zeta_1\sqrt{V_1}+C.
\end{align}
This finishes the proof of \eqref{lyapunovcondition} when $n=1$.

For $n>1$, we apply $\CL$ to $V_1^n$ and obtain the identity

\[\CL V_1^n=nV^{n-1}\mathcal{L}V_1+n(n-1)V_1^{n-2}\Big( \gamma \beta^{-1}\Big|  2H_1\frac{\epsilon p}{\sqrt{1+\epsilon|p|^2}}+\epsilon q   \Big|^2 +\sum_{i=1}^M \alpha_i\beta^{-1}\left|2H_1\epsilon z_i \right|^2 \Big).\]
Since $\frac{\epsilon |p|}{\sqrt{1+\epsilon|p|^2}}\leq 1$ for all $\epsilon\in(0,1)$ and by Assumption \ref{UAssumptions}, we have
\begin{align*}
    \Big|  2H_1\frac{\epsilon p}{\sqrt{1+\epsilon|p|^2}}+\epsilon q   \Big|^2\leq CH_1^2+C.
\end{align*}
Furthermore, by definition of $H_1$, we can estimate 

\begin{align*}
    \sum_{i=1}^M|2H_1\epsilon z_i|^2\leq 4\epsilon H_1^3.
\end{align*}
By applying \eqref{Vbound}, we have

\begin{align*}
    V_1^\frac{3}{2}\geq \zeta_1H_1^3-C.
\end{align*}
On the other hand, by applying \eqref{Lvbound}, we get
\begin{align*}
    nV^{n-1}\CL V_1\leq -\zeta_1 V_1^{n-\frac{1}{2}}+CV_1^{n-1}.
\end{align*}
Putting the above together, we obtain

\begin{align*}
\CL V_1^n &\leq -\zeta_1 V_1^{n-\frac{1}{2}}+ CV_1^{n-1}+n(n-1)V_1^{n-2}\Big(CV_1 +\epsilon \zeta _1 V_1^\frac{3}{2}+C\Big)\\
&\leq -\frac{\zeta_1-\epsilon \zeta_1n(n-1)}{2}V_1^{n-\frac{1}{2}}+C.
\end{align*}
Since $\zeta_1$ is indpendent of $\epsilon$, for $\epsilon$ sufficiently small dependent on $n\in\IN$, \eqref{lyapunovcondition} holds, thus completing the proof of this proposition.

\end{proof}

 We remark that the above proof does not hold for the case $\gamma=0$ since the dissipation term involving $p$ on the right hand side of \eqref{LHinequality} is then cancelled. Therefore, to find the Lyapunov function for this case, we will add an additional perturbation to $V_1$, as well as redefine our choice of scalar multiples for these terms. It is also important to note that we have to further restrict the construction to Lipschitz $\nabla U$.

\begin{proposition}
\label{lyaponovconditiongammaeq0}
   Let $\gamma = 0$, and suppose that $\nabla U$ is Lipschitz. Let $V_2$ be defined by,
\begin{align} \label{form:V_2}
V_2(q,p,z)=H_1(q,p,z)^2+\epsilon_1\sum_{i=1}^M\langle p,z_i\rangle+\epsilon_2\langle q,p\rangle+ \kappa,
\end{align}
where $H_1$ is defined as in \eqref{hamiltonianforlyapunovfunction}, $\epsilon_1=\Lambda_1\epsilon$ and $\epsilon_2=\Lambda_2\epsilon_1$ for $\Lambda_1, \Lambda_2>0$. Then, for all $n\in\IN$, for sufficiently small $\epsilon ,\Lambda_1,\Lambda_2>0$, with  $\epsilon=\epsilon(n,\gamma,\alpha, \lambda,M)$ sufficiently small, and $\Lambda_1=\Lambda_1(\alpha,\lambda)$ and $\Lambda_2=\Lambda_2(\alpha,\lambda)$ sufficiently small, independent of $\epsilon$, we have
\begin{equation} \label{lyapunovcondition:gamma=0}
    \CL V_2^n\leq -\zeta V_2^{n-\frac{1}{2}}+C,
\end{equation}
for some constants $\zeta=\zeta(\epsilon,n,\Lambda_1,\Lambda_2,M) ,\, C=C(\epsilon,n,\Lambda_1,\Lambda_2,M)>0$.
\end{proposition}

\begin{proof}
In the below, $C,\zeta_1>0$ may change from line to line, and $\zeta_1$ is independent of $\epsilon$. 
When $\gamma=0$, from \eqref{infinitesimaloperatorH}, the generator of \eqref{GLEpqz} acting on $H_1$ is given by
\begin{align}
       \mathcal{L}H_1&=-\epsilon\sum_{i=1}^M\alpha_i| z_i|^2+\epsilon d\sum_{i=1}^M\alpha\beta^{-1}. \label{infinitesimaloperatorH0gamma}
\end{align}
We compute
\begin{align}
\mathcal{L}H_1^2=&2H_1\mathcal{L}H_1+\epsilon^2\sum_{i=1}^M2\alpha_i\beta^{-1}|z_i|^2.\label{infinitesimaloperatorH20gamma}
  \end{align}
Similar to the proof of Proposition \ref{lyaponovconditiongammageq0}, we first consider the case when $n=1$. By combining \eqref{infinitesimaloperatorH0gamma} and \eqref{infinitesimaloperatorH20gamma}, along with \eqref{H1inequality}, we can estimate
\begin{align}
    \CL H_1^2\leq-\epsilon^\frac{3}{2}\sum_{i=1}^M\alpha_i|p||z_i|^2-\epsilon^2\sum_{i=1}^M\alpha_i|z_i|^4-{\epsilon}\sum_{i=1}^M\alpha_i| z_i|^2+\epsilon d\sum_{i=1}^M2\alpha_i\beta^{-1} H_1+\epsilon^2\sum_{i=1}^M2\alpha_i\beta^{-1}|z_i|^2.
    \label{zerogammaLH2bound}
\end{align}
Concerning the fist inner product term of \eqref{form:V_2}, we compute
\begin{align*}
\CL \langle p,z_i  \rangle &= -\langle \lambda_i\nabla K (p),p  \rangle  -\langle  \alpha_iz_i,p  \rangle  - \langle \nabla U(q),z_i\rangle+  \langle z_i, \sum_{j=1}^M  \lambda_j z_j   \rangle,
\end{align*}
for $i=1,...,M$. Note that the Cauchy-Schwarz inequality and the Lipschitz continuity of $\nabla U$ yield,
\[ -\epsilon\langle \nabla U(q),z_i\rangle\leq \epsilon \zeta_1(1+|q|)|z|.\]
By applying \eqref{innerproductpkpbound}, we get

\begin{align*}
\epsilon_1\mathcal{L}\langle p,z_i  \rangle    
&\leq {-\Lambda_1\lambda_i\sqrt{1+\epsilon|p|^2}}+\Lambda_1\epsilon\alpha_i|z_i||p| +  \epsilon\zeta_1|q|^2+ \epsilon a_1|z_i|+ \epsilon \Lambda_1\zeta_1|z_i|^2+  \Lambda_1 \epsilon\langle z_i \sum_{i=1}^M  \lambda_i z_i   \rangle  + C.
\end{align*}
Since
\[|z_i|+|z_i|^2\leq 2|z_i|^2+1,\]
we find
\begin{align}
\epsilon_1\sum_{i=1}^M\langle p,z_i  \rangle    &\leq {-\Lambda_1\sum_{i=1}^M\lambda_i\sqrt{1+\epsilon|p|^2}}+\Lambda_1\epsilon\sum_{i=1}^M\alpha_i|z_i||p|  \notag\\&+\epsilon\zeta_1|q|^2+ \epsilon \Lambda_1\zeta_1\sum_{i=1}^M|z_i|^2+  \epsilon M\Lambda_1\sum_{i=1}^M \lambda_i|z_i|^2 + C.  \label{zerogammapzbound}
\end{align}
Applying Assumption \ref{UAssumptions}, and choosing $\Lambda_2\leq \frac{1}{8}\lambda_i$ for all $1\leq i\leq M$, we have
\begin{align}
\epsilon_2\mathcal{L}\langle q,p\rangle&=\epsilon_2\langle p,\nabla K(p)\rangle-\epsilon_2\langle q,\nabla U(q)\rangle +\epsilon_2\sum_{i=1}^M\lambda_i\langle q,z_i\rangle \notag\\
    &\leq \epsilon_2 \frac{|p|^2}{\sqrt{1+\epsilon|p|^2}}-{\epsilon_2} \zeta _1|q|^{\lambda+1}+\sum_{i=1}^M\lambda_i|q|^2+\epsilon_2^2\sum_{i=1}^M\lambda_i|z_i|^2+C \notag\\
    &\leq \frac{\Lambda_1}{8}\sum_{i=1}^M\lambda_i\sqrt{\epsilon}|p|-{\Lambda_1}\Lambda_2\epsilon \zeta_1|q|^{\lambda+1}+\epsilon^2\sum_{i=1}^M\lambda_i|z_i|^2+C.  \label{zerogammaqpbound}
\end{align}
Putting together \eqref{zerogammaLH2bound}, \eqref{zerogammapzbound}, and \eqref{zerogammaqpbound},
we obtain
\begin{align*}
    \mathcal{L}V_2\leq &\left(\epsilon d\sum_{i=1}^M2\alpha_i\beta^{-1} H_1 + \epsilon\sum_{i=1}^M\left(\Lambda_1M\zeta_1-\alpha_i\right)|z_i|^2-\frac{\epsilon}{4}\zeta_1|q|^{\lambda+1}{-\frac{1}{2}\Lambda_1\sum_{i=1}^M\lambda_i\sqrt{1+\epsilon|p|^2}} \right)\\
    &-\frac{1}{8}\Lambda_1\sum_{i=1}^M\lambda_i\sqrt{\epsilon}{|p|}+\sum_{i=1}^M\left(\epsilon\Lambda_1\alpha_i|z_i||p|-\frac{1}{4}\Lambda_1\lambda_i\sqrt{\epsilon}{|p|}-\epsilon^\frac{3}{2}\alpha_i|p||z_i|^2\right)\\
    &+\epsilon^2\sum_{i=1}^M\left( -\alpha_i|z_i|^4+ \lambda_i|z_i|^2+2\alpha_i\beta^{-1}|z_i|^2\right)+\left(\epsilon\zeta_2|q|^2-\epsilon{\Lambda_1}\Lambda_2 \zeta_1|q|^{\lambda+1}\right)+C.
\end{align*}
Choosing $\Lambda_1\leq \frac{\lambda_i}{\alpha_i}$ for all $1\leq i\leq M$, we can estimate 
\[ \epsilon\Lambda_1\alpha_i|z_i|-\frac{1}{4}\Lambda_1\lambda_i\sqrt{\epsilon}-\epsilon^{\frac{3}{2}}\alpha_i|z_i|^2\leq0.\]
Thus, by applying Assumption \ref{UAssumptions}, proceeding similarly to the proof of Proposition \ref{lyaponovconditiongammageq0} and choosing $\Lambda_1$ sufficiently small, and then we find that
\begin{align*}
\mathcal{L}V_2\leq-\zeta H_1^2+C.
\end{align*}
By choosing $\Lambda_1\leq \frac{1}{2}$, we can choose $\kappa>0$ such that
\begin{equation}
\label{V2bound}
    \zeta_1 H_1^2-C\leq V_1\leq C_1H_1^2+C.
\end{equation}
By combining the above, we find
\[\mathcal{L}V_2\leq-\zeta_1\sqrt{V_2}+C,\]
which proves \eqref{lyapunovcondition:gamma=0} for the case $n=1$.

Turning to the case $n>1$, note that,
\[\CL V^n=nV^{n-1}\mathcal{L}V+n(n-1)V^{n-2}\Big( \sum_{i=1}^M \alpha_i\beta^{-1}\left|2H_1\epsilon z_i +\epsilon p\right|^2 \Big).\]
The rest of the proof follows closely the proof of Proposition \ref{lyaponovconditiongammageq0} when $n> 1$, and thus is omitted.
\end{proof}

\subsection{Hörmander's and Solvability Conditions}
\label{Polynomial Ergodicity}

We now prove Hörmander's condition and the solvability condition of the GRLE, in order to show the GRLE satisfies the assumptions of Theorem \ref{hairerbound}.

To prove Hörmander's condition, we define the family of vector fields,
\begin{subequations}
\label{vectorfields}
    \begin{align}
     X_0&= \langle\nabla_q ,\nabla K (p)\rangle+\langle \nabla_p,-\nabla U(q)-\gamma \nabla K (p)+\sum_{j=1}^m\lambda_jz_j \rangle+ \sum_{j=1}^m\langle  \nabla_{z_j}, -\lambda_j\nabla K (p)-\alpha_jz_j  \rangle,\\
     X_{q_i}&=0,\\
     X_{p_i}&=\sqrt{2\beta^{-1}\gamma}\frac{\delta}{\delta {p_i}},\\
     X_{z_{ki}}&=\sqrt{2\beta^{-1}\alpha_k}\frac{\delta}{\delta z_{ki}}.
 \end{align}
 \end{subequations}
Define the family of Lie algebras recursively as
\begin{align*}
    \mathcal{A}_0&=\Lie\left(X_{q_i},X_{p_i},X_{z_{ki}}\right)_{i=1,...,d\text{, }k=1,...,m},\\
    \mathcal{A}_n&=\Lie\left([X_0,X]\right)_{X\in \mathcal{A}_{n-1}},
\end{align*}
for $n\in\IN$, where the Lie brackets denote
\[[X,Y]=X(Y)-Y(X).\]
Further we define
\[\mathcal{A}=\Lie(\mathcal{A}_n)_{n\in\IN_0}.\]
Then the following Proposition holds:

\begin{proposition}
\label{Hörmanderscondition}
The family of vector fields \eqref{vectorfields} satisfies Hörmander's condition: for all $x\in\IR^{(2+M)d}$, the vector space,
\[\text{span}\left\{X(x):X\in\mathcal{A} \right\}=\IR^{(2+M)d}.\]

\end{proposition}

\begin{proof}
Direct computations give
\[  \frac{1}{\sqrt{2\beta^{-1}\alpha_k}}\left[X_0,X_{z_{ki}}\right]=-\lambda_k\frac{\delta}{\delta p_i}+\alpha_k\frac{\delta}{\delta z_{ki}}.\]
Thus, $\frac{\delta}{\delta p_i}\in\mathcal{A}$,  for any choice of $\gamma\geq 0$. Furthermore, we have
\begin{align*}
    \left[X_0, \frac{\delta}{\delta p_i} \right]=&-\left( \frac{(1+\epsilon|p|^2)\frac{\delta}{\delta q_i}-\epsilon p_i\sum_{j\leq d}p_j\frac{\delta}{\delta q_j}}{(1+\epsilon|p|^2)^\frac{3}{2}} \right)+\gamma\left( \frac{(1+\epsilon|p|^2)\frac{\delta}{\delta p_i}-\epsilon p_i\sum_{j\leq d}p_j\frac{\delta}{\delta p_j}}{(1+\epsilon|p|^2)^\frac{3}{2}} \right)\\+ &\sum_{l=1}^m\lambda_l\left( \frac{(1+\epsilon|p|^2)\frac{\delta}{\delta z_{li}}-\epsilon p_i\sum_{j\leq d}p_j\frac{\delta}{\delta z_lj}}{(1+\epsilon|p|^2)^\frac{3}{2}} \right).
\end{align*}
So that, for $i=1,..,d$, we have
\begin{align*}
    \frac{1}{\sqrt{2\beta^{-1}\alpha_k}}\left[X_0,\left[X_0,X_{z_{k_i}}\right] \right]=&\lambda_k\left( \frac{(1+\epsilon|p|^2)\frac{\delta}{\delta q_i}-\epsilon p_i\sum_{j\leq d}p_j\frac{\delta}{\delta q_j}}{(1+\epsilon|p|^2)^\frac{3}{2}} \right)-\gamma\lambda_k\left( \frac{(1+\epsilon|p|^2)\frac{\delta}{\delta p_i}-\epsilon p_i\sum_{j\leq d}p_j\frac{\delta}{\delta p_j}}{(1+\epsilon|p|^2)^\frac{3}{2}} \right)\\- &\lambda_k\sum_{l=1}^m\lambda_l\left( \frac{(1+\epsilon|p|^2)\frac{\delta}{\delta z_{li}}-\epsilon p_i\sum_{j\leq d}p_j\frac{\delta}{\delta z_lj}}{(1+\epsilon|p|^2)^\frac{3}{2}} \right)-\alpha_k\lambda_k\frac{\delta}{\delta p_i}+\alpha_k^2\frac{\delta}{\delta z_{ki}}.
\end{align*}
From this, it follows that
\begin{equation}
\label{c1}
\frac{(1+\epsilon|p|^2)\frac{\delta}{\delta q_i}-\epsilon p_i\sum_{j\leq d}p_j\frac{\delta}{\delta q_j}}{(1+\epsilon|p|^2)^\frac{3}{2}}\in\mathcal{A}.    
\end{equation}
For all $l\in1,...,d$, there exists $a\in\IR^d$ such that
\begin{equation}
\label{c2}
\sum_{i=1}^d a_i\left(\left(1+\epsilon|p|^2\right)\frac{\delta}{\delta q_i}-\epsilon p_i\sum_{j\leq d}p_j\frac{\delta}{\delta q_j} \right)=\frac{\delta}{\delta q_l}.    
\end{equation}
This statement follows from the fact that the matrix:

\[\bf{B}=\begin{pmatrix}
    -\frac{1}{\epsilon}-|p|^2+p_1^2 & p_2p_1 & \cdots &p_dp_1\\
    p_1p_2 & -\frac{1}{\epsilon}-|p|^2+p_2^2 & \cdots &p_dp_2\\
    \vdots & \vdots & \ddots & \vdots\\
    p_1p_d & p_2p_d & \cdots & -\frac{1}{\epsilon}-|p|^2+p_d^2,
\end{pmatrix}\]
satisfies the conditions of Lemma \ref{nonsingularmatrix}, and so is invertible. Thus we can choose $a$ as,
\[a=-\frac{1}{\epsilon}{\bf{B}^{-1}}e_l.\]

We deduce from \eqref{c1} and \eqref{c2} that $\frac{\delta}{\delta q_l}\in \mathcal{A}$, which completes the proof.
\end{proof}

The solvability condition is stated as follows:
\begin{proposition}
\label{controlproblemstatement}
    The control problem associated with \eqref{GRLEpqz} is solvable: for all $(q_0,p_0,z_{1,0},...,z_{M,0})\in\IR^{2+M}$, there exists a smooth collection of paths $U,U_1,...,U_k$, which satisfy the ordinary differential equation
\begin{subequations}
\label{controlproblem}
 \begin{align}
 \df q(t)&= \nabla K (p(t))\,\df t,\\
\df p(t)&=-\nabla U(q(t))\,\df t-\gamma \nabla K (p(t))\,\df t+\sum_{i=1}^M \lambda_i z_i(t)\,\df t+\sqrt{2\gamma\beta^{-1}}\, \df U(t),\\
\df z_i(t)&=-\lambda_i\ \nabla K (p(t))\,\df t-\alpha_i z_i(t)\,\df t+\sqrt{2\alpha_i\beta^{-1}}\,\df U_i(t),
\end{align}   
\end{subequations}
subject to boundary conditions
\begin{subequations}
\label{initialconditionsforcontrolproblem}
    \begin{align}
    (q(0),p(0),z_1(0),...,z_M(0))&=(q_0,p_0,z_{1,0},...,z_{M,0}),\\
    (q(T),p(T),...,z_1(T),...,z_M(T))&=\bf{0},
\end{align}
\end{subequations}
for some $T\geq 0$.
\end{proposition}

\begin{proof}
\textbf{Case 1 ($\gamma>0$):} The result follows from Lemma 3.7 from \cite{duong_trend_2024}. Indeed, we can find $q\in \text{C}^{\infty}(\IR)$ such that the path of $p$, given by
\begin{equation}
    p(s)=\frac{q'(s)}{\sqrt{1-\epsilon|q'(s)|^2}},\quad 0\leq s\leq T, \label{rearrangingtogetp}
\end{equation}
is well defined, that is
\begin{align}
\label{boundednessoffirstderivative}
    |q'(s)|<{\frac{1}{\sqrt{\epsilon}}},\quad 0\leq s\leq T,
\end{align}
and that $q$ and $p$ satisfy initial conditions. We then define $z_k\in \text{C}^{\infty}(\IR)$ by

\[z_k(s)=\begin{cases}
    z_{k,0}, &0\leq s< \rho,\\
    \text{monotonicity,} &\rho\leq s\leq T-\rho,\\
    0, &T-\rho <s\leq T,
\end{cases}\]
for some $\rho>0$. We then define
\[U(t)=\frac{1}{\sqrt{2\gamma\beta^{-1}}}\int_0^t\frac{\df p(s)}{\df s}+\nabla U(q(s)) + \gamma\nabla K (p(s))- \sum_{i=1}^M\lambda_iz_i(s)\, \df s, \]
and for $i=1,\ldots, M$
\begin{align}
\label{Uk}
    U_i(t)=\frac{1}{\sqrt{2\alpha_i\beta^{-1}}}\int_0^t\frac{\df z_i(s)}{\df s}+\lambda_i\nabla K (p(s))+\alpha_iz_i(s)\, \df s.
\end{align}

\textbf{Case 2 ($\gamma=0$):} The control problem reduces to:
\begin{subequations}
\label{controlproblemzerogamma}
 \begin{align}
 \df q(t)&= \nabla K (p(t))\,\df t, \label{controlproblemqgamma0}\\
\df p(t)&=-\nabla U(q(t))\,\df t+\sum_{i=1}^M \lambda_i z_i(t)\,\df t,  \label{controlproblempgamma0}\\
\df z_k(t)&=-\lambda_k\ \nabla K (p(t))\,\df t-\alpha_k z_k(t)\,\df t+\sqrt{2\alpha_k\beta^{-1}}\,\df U_k(t).
\end{align}   
\end{subequations}
It suffices to choose a path for $q$, $p$ and $z$ which satisfies \eqref{controlproblemqgamma0} and \eqref{controlproblempgamma0}.\\
From \eqref{controlproblemqgamma0} and \eqref{controlproblempgamma0}, we have
\begin{subequations}
\label{controlproblemz}
    \begin{align}
\sum_{i=1}^M\lambda_iz_i(s)&=p'(s)+\nabla U(q(s)),\\
p'(s)&=\frac{q''(s)\left(1-\epsilon|q'(s)|^2\right)+\epsilon q'(s)\left\langle q'(s),q''(s)\right\rangle}{\left(1-\epsilon|q'(s)|^2\right)^{\frac{3}{2}}},
\end{align}
\end{subequations}
for $0\leq s\leq T$. Thus, we must satisfy the conditions
\begin{subequations}
\label{conditionsforcontrolproblem}
    \begin{align}
-\nabla U(0)&=q''(T),\\
q'(T)&=0.\\
q'(0)&=\nabla K (p(0)).
\end{align}
\end{subequations}
For simplicity of notation, let $k_p:=K (p(0))$. In order for $p$ to be well defined, we must again satisfy \eqref{boundednessoffirstderivative}.\\
There exists $q''_0\in\IR^d$ such that $q''(0)=q''_0$ and satisfies \eqref{controlproblemz}. Indeed, $q''_0$ satisfies the simultaneous equation
\begin{align}
\label{simulaneousequation}
\sum_{i=1}^M\lambda_iz_{i,0}=\frac{q''_0(1-\epsilon|k_{p}|^2)+\epsilon k_{p}  \left\langle k_{p},q''_0\right\rangle}{\left(1-\epsilon|k_{p}|^2\right)^\frac{3}{2}}+\nabla U(q_0).
\end{align}
In matrix form, \eqref{simulaneousequation} is given by

\[{\bf{B}}q''(0)=\frac{1}{\epsilon}{\left(1-\epsilon|k_{p,0}|^2\right)^\frac{3}{2}}\left(\sum^{m}_{i=1}\lambda_iz_{i,0}-\nabla U(q_0)\right),\]
where $\bf{B}$ is given by,

\[{\bf{B}}=\begin{pmatrix}
    \frac{1}{\epsilon}-|k_p|^2+k_{p_{1}}^2 & k_{p_1}k_{p_2} &\cdots &k_{p_1}k_{p_d}\\
    k_{p_2}k_{p_1} & \frac{1}{\epsilon}-|k_p|^2 +k_{p_2}^2 &\cdots &k_{p_2}k_{p_d}\\
    \vdots & \vdots & \ddots & \vdots \\
    k_{p_n}k_{p_1} & k_{p_n}k_{p_2} & \cdots &  \frac{1}{\epsilon}-|k_p|^2+k_{p_{d}}^2
\end{pmatrix}.\]
By Lemma \ref{nonsingularmatrix}, $\bf{B}$ is invertible, thus we can find initial condition $q''(0)$ for the control problem.\\
Next we seek $Q\in C^{\infty}(\IR)$ so that
\[q(t)=q_0+\int_0^tQ(s)\,\df s.\]
We can construct $Q$ as follows: let $\psi$ be a smooth step function. By appropriately translating and scaling $\psi$, there exist $\psi_0,\psi_1\in C^{\infty}(\IR)$ and $t_1,t_3>0$ such that, for all $n\in\IN$,
\[\psi_0(0)=k_p,\quad \psi_0(t_1)=a,\quad \psi_0'(0)=q''_0,\quad \psi_0^{(n)}(t_1)=0,\]
where $0<|a|\leq\frac{1}{\sqrt{\epsilon}}$, and,
\[\psi_1(0)=\text{sgn}(\nabla U(0))|a|,\quad \psi_1(t_3)=0,\quad \psi_1^{(n)}(0)=0,\quad \psi_1'(t_3)=-\nabla U(0).\]
For any $A\in\IR$, and $a_1,a_2\in\IR$, there exist $g\in C^{\infty}(\IR)$ and $t_2>0$ such that, for all $n\in\IN$,

\[g(0)=a_1,\quad g(t_2)=a_2, \quad g^{(n)}(0)=g^{(n)}(t_2)=0,\quad \int_0^{t_2}g(s)\df s=A.\]
Let $T_1=t_1$, $T_2=t_1+t_2$ and $T=t_1+t_2+t_3$. Choosing $a_1=\psi_0(0)$, $a_2=\psi_1(0)$ and $A=-q_0-\int_0^{t_1}\psi_0(s)\,\df s-\int_0^{t_3}\psi_1(s)\,\df s$, we choose,

\[Q(t)=\begin{cases}
    \psi_0(t), &t\in[0,T_1],\\
    g(t-T_1), &t\in(T_1,T_2],\\
    \psi_1(t-T_2), &t\in(T_2,T].
\end{cases}\]
By our choice of $Q$, we have found $q$ which satisfies \eqref{initialconditionsforcontrolproblem}, \eqref{boundednessoffirstderivative}, and \eqref{conditionsforcontrolproblem}, as well as $q''(0)=q''_0$. $p$ and $z$ can then be chosen according to \eqref{rearrangingtogetp} and \eqref{controlproblemz} respectively. By choosing $U_k$ as defined in \eqref{Uk}, we have found a solution to the control problem.
\end{proof}

\subsection{Proof of Theorem \ref{polynomialerogicity}} Having established the crucial properties in the previous subsections, we are now in a position to conclude Theorem \ref{polynomialerogicity}, by verifying the hypothesis of Theorem \ref{hairerbound}.
\begin{proof}[Proof of Theorem \ref{polynomialerogicity}]
    By Propositions \ref{lyaponovconditiongammageq0} and \ref{lyaponovconditiongammaeq0}, there exists $\zeta>0$ such that

    \[\CL V(x)^n\leq-\zeta V(x)^{n-\frac{1}{2}}.\] 
      It follows that, for all $n\in\IN$, for sufficiently small $\epsilon>0$, $V^n$ is a $\phi$- Lyapunov function, where $\phi(t)=\zeta t^{1-\frac{1}{2n}}$. From Theorem \ref{hairerbound}, as well as Propositions \ref{Hörmanderscondition} and \ref{controlproblemstatement}, it follows that,

    \[\|P_t(x,\cdot)-\rho_\beta\|_{\text{TV}}\leq \frac{\zeta_0 V(x)^n}{\zeta\left(\frac{t}{2n\zeta}+1\right)^{2n-1}}.\]
The result follows by choosing $r=2n+1$.    
\end{proof}

\section{White Noise and Newtonian Limits}
In this section, we prove Theorem \ref{relativisticwhitenoiselimit} and Theorem \ref{relativisticnewtonianlimit} on the white-noise limit and Newtonian limit of the GRLE system \eqref{GRLEpqz}, respectively. This will be completed in three steps. We first show that the GRLE has bounded moments. We then prove the limits given that $\nabla U$ is Lipschitz continuous. Finally, we combine these results to show that this holds generally for any $U\in C^{\infty}(\IR)$. This method of proof is adapted from \cite{duong_trend_2024, nguyen2018small}.

\label{White Noise and Newtonian Limits}
\subsection{Moments bound}
We first show the boundedness of the moments of the GRLE systems.
\begin{proposition}
\label{boundedprocessesgen}

Let $(q,p,z)$, $(q^{(c)},p^{(c)})$ and $(q^{(\varepsilon)},p^{(\varepsilon)},z^{(\varepsilon)})$ be solutions of \eqref{GRLEpqz}, \eqref{RLE}, and \eqref{GRLEpqzrescaled} respectively. Then under Assumptions \ref{UAssumptions} and \ref{boundedinitialconditionsrel}, for all $n\geq1$ and $T>0$, the following estimates hold
 \begin{equation}
 \label{GLEExpectationbound}
\IE\left[\sup_{t\in[0,T]}\left|p(t)\right|^n+\sup_{t\in[0,T]}\left|q(t)\right|^n+\sum_{i=1}^M\sup_{t\in[0,T]}\left|z_i(t)\right|^n\right]<C,     
 \end{equation}

  \begin{equation}
 \label{RLEExpectationbound}
\IE\left[\sup_{t\in[0,T]}\left|p^{(c)}(t)\right|^n+\sup_{t\in[0,T]}\left|q^{(c)}(t)\right|^n\right]<C;     
 \end{equation}
 and,
 \begin{equation}
  \label{GLREExpectationbound}
     \IE\left[\sup_{t\in[0,T]}\left|p^{(\varepsilon)}(t)\right|^n+\sup_{t\in[0,T]}\left|q^{(\varepsilon)}(t)\right|^n+\sum_{i=1}^M\sup_{t\in[0,T]}\left|z^{(\varepsilon)}_i(t)\right|^n\right] <C,
 \end{equation}
for some constant $C=C(T,n)>0$ independent of $\varepsilon\in(0,1]$ and $\epsilon=1/c^2\in(0,1]$.
\end{proposition}
\begin{proof}
We first prove \eqref{GLREExpectationbound}.  Let $(q^{(\varepsilon)},p^{(\varepsilon)},z^{(\varepsilon)})$ denote the solution to \eqref{GRLEpqzrescaled}. In the below, $C$ is independent of $c$ and $\varepsilon$ and may change from line to line. By applying Itô's formula to $e^{-\frac{\alpha}{\varepsilon}t}z_i^{(\varepsilon)}(t)$, we find from \eqref{zrescaled} that
\begin{align}
\label{zito}
    z_i^{(\varepsilon)}(t)=e^{-\frac{\alpha_i}{\varepsilon}t}z_i^{(\varepsilon)}(0)+\int_0^te^{-(t-s)\frac{\alpha_i}{\varepsilon}}\left(-\frac{\lambda_i}{\sqrt{\varepsilon}}\nabla K (p^{(\varepsilon)}(s))\, \df s+\sqrt{\frac{2\alpha_i}{\varepsilon}}\, \df W_i(s)\right).
\end{align}
Substituting this expression into \eqref{prescaled} we get
\begin{align}
    p^{(\varepsilon)}(t)=&- \int_0^t\nabla U(q^{(\varepsilon)}(s))\,\df t-\gamma\int_0^t\nabla K (p^{(\varepsilon)}(s))\,\df s+p^{(\varepsilon)}(0)\notag\\ &+\frac{1}{\sqrt{\varepsilon}}\sum_{i=1}^M\lambda_i     \left(\int_0^te^{-\frac{\alpha_i}{\varepsilon}s}z_i^{(\varepsilon)}(0)\,\df s  \right.\notag \\ &\left. +\frac{1}{\sqrt{\varepsilon}}\int_0^t\int_0^se^{\frac{-(s-u)\alpha_i}{\varepsilon}}\left(-\lambda_i \nabla K(p^{(\varepsilon)}(u))\, \df u+\sqrt{2\alpha_i}\,\df W_i(u)\right)\,\df s     \right)\notag\\
    &+\int_0^t\sqrt{2\gamma}\, \df W(s).\label{prescaled2}
\end{align}
We will use the following elementary identity
\begin{equation}
\label{OUsolution}
    \frac{1}{a}\int_0^t\Big(1-e^{-a(t-s)}\Big)\, \df W(s)=\int_0^t\int_0^se^{-a(t-u)}\, \df W(u)\, \df s.
\end{equation}
By application of \eqref{OUsolution}, as well as Fubini's theorem, we have
\begin{align*}
    &\int_0^t\int_0^se^{\frac{-(s-u)\alpha_i}{\varepsilon}}\left(-\lambda_i \nabla K(p^{(\varepsilon)}(u))\, \df u+\sqrt{2\alpha_i}\,\df W_i(u)\right)\,\df s\\=&\frac{\varepsilon}{\alpha_i}\int_0^t\left(1-e^{-(t-u)\frac{\alpha_i}{\varepsilon}}\right)\left(-\lambda_i \nabla K(p^{(\varepsilon)}(u))\, \df u+\sqrt{2\alpha_i}\,\df W_i(u)\right).
\end{align*}
Substituting this into \eqref{prescaled2}, we get
\begin{align*}
    p^{(\varepsilon)}(t)=&- \int_0^t\nabla U(q^{(\varepsilon)}(s))\,\df t-\gamma\int_0^t\nabla K (p^{(\varepsilon)}(s))\,\df s+p^{(\varepsilon)}(0)+\frac{1}{\sqrt{\varepsilon}}\sum_{i=1}^M\lambda_i     \int_0^te^{-\frac{\alpha_i}{\varepsilon}s}z^{(\varepsilon)}_i(0)\,\df s   \\ & +\frac{\lambda_i}{\alpha_i}\int_0^t\left(1-e^{-(t-s)\frac{\alpha_i}{\varepsilon}}\right)\left(-\lambda_i \nabla K (p^{(\varepsilon)}(s))\, \df s+\sqrt{2\alpha_i}\,\df W_i(s)\right)     \\
    &+\int_0^t\sqrt{2\gamma}\, \df W(s).
\end{align*}
We define
\[\Gamma_1(t)=\sqrt{1+\epsilon|p^{(\varepsilon)}(t)|^2}U(q^{(\varepsilon)}(t))+\frac{1}{2}|p^{(\varepsilon)}(t)|^2+\frac{1}{2}\epsilon U(q^{(\varepsilon)}(t))^2.\]
By applying Itô's formula, we can write this term as
\begin{align*}
    \Gamma_1(t)=\Gamma_1(0)
    +&\int_0^t\Big\langle \frac{\epsilon p^{(\varepsilon)}(s)}{\sqrt{1+\epsilon|p^{(\varepsilon)}(s)|^2}}U(q^{(\varepsilon)}(s))+p^{(\varepsilon)}(s),\gamma \nabla K(p^{(\varepsilon)}(s))\\
    &\hspace{2cm}+ \sum_{i=1}^M \frac{\lambda_i}{\sqrt{\varepsilon}}e^{-\frac{\alpha_i}{\varepsilon}s}z^{(\varepsilon)}_i(0) -\frac{\lambda_i^2}{\alpha_i}(e^{-(t-s)\frac{\alpha_i}{\varepsilon}}+1)\nabla K(p^{(\varepsilon)}(s)) \Big\rangle\, \df s\\
    +&\int_0^t\Big(\gamma+\sum_{i=1}^M\frac{\lambda_i^2}{\alpha_i}\left(1-e^{-(t-s)\frac{\alpha_i}{\varepsilon}}\right)^2\Big)\Big(\frac{\epsilon dU(q^{(\varepsilon)}(s))}{(1+\epsilon|p^{(\varepsilon)}(s)|^2)^{\frac{3}{2}}}+d\Big)\,\df s\\
    + &\int_0^t\sqrt{2\gamma}\Big\langle \frac{\epsilon p^{(\varepsilon)}(s)U(q^{(\varepsilon)}(s))}{\sqrt{1+\epsilon|p^{(\varepsilon)}(s)|^2}} +p^{(\varepsilon)}(s),\df W(s) \Big\rangle\\
    +    &\sum_{i=1}^M\int_0^t  \sqrt{\frac{{2\lambda_i^2}}{{\alpha_i}}}\left   (e^{(t-s)\frac{\alpha_i}{\varepsilon}}+1\right)\Big\langle \frac{\epsilon p^{(\varepsilon)}(s)U(q^{(\varepsilon)}(s))}{\sqrt{1+\epsilon|p^{(\varepsilon)}(s)|^2}} +p^{(\varepsilon)}(s),\df W_i(s) \Big\rangle.
\end{align*}
Using the following estimates
\begin{align*}
&\int_0^t\Big\langle \frac{\epsilon p^{(\varepsilon)}(s)}{\sqrt{1+\epsilon|p^{(\varepsilon)}(s)|^2}}U(q^{(\varepsilon)}(s))+p^{(\varepsilon)}(s), \sum_{i=1}^M \frac{\lambda_i}{\sqrt{\varepsilon}}e^{-\frac{\alpha_i}{\varepsilon}s}z_i^{(\varepsilon)}(0)  \Big\rangle\, \df s   \\
\leq &\int_0^t \Big| \frac{\epsilon p^{(\varepsilon)}(s)}{\sqrt{1+\epsilon|p^{(\varepsilon)}(s)|^2}}U(q^{(\varepsilon)}(s))+p^{(\varepsilon)}(s) \Big|^2+\sum_{i=1}^M\Big|\frac{\lambda_i}{\sqrt{\varepsilon}}e^{-\frac{\alpha_i}{\varepsilon}s}z^{(\varepsilon)}_i(0)\Big|^2\df s \\
\leq &\int_0^t\Big|\frac{\epsilon p^{(\varepsilon)}(s)}{\sqrt{1+\epsilon|p^{(\varepsilon)}(s)|^2}}U(q^{(\varepsilon)}(s))+p^{(\varepsilon)}(s)\Big|^2 \df s+\sum_{i=1}^M \left| z_i^{(\varepsilon)}(0) \right|^2 \int_0^t\frac{\lambda_i^2}{\varepsilon}e^{-\frac{2\alpha}{\varepsilon}}\df s\\
\leq &\int_0^t\Big|\frac{\epsilon p^{(\varepsilon)}(s)}{\sqrt{1+\epsilon|p^{(\varepsilon)}(s)|^2}}U(q^{(\varepsilon)}(s))+p^{(\varepsilon)}(s)\Big|^2 \df s + C\sum_{i=1}^M\left|z_i^{(\varepsilon)}(0) \right|^2,
\end{align*}
we have
\begin{align*}
    &\int_0^t\Big\langle \frac{\epsilon p^{(\varepsilon)}(s)}{\sqrt{1+\epsilon|p^{(\varepsilon)}(s)|^2}}U(q^{(\varepsilon)}(s))+p^{(\varepsilon)}(s),\gamma \nabla K(p^{(\varepsilon)}(s))\\
    &\hspace{2cm}+ \sum_{i=1}^M \frac{\lambda_i}{\sqrt{\varepsilon}}e^{-\frac{\alpha_i}{\varepsilon}s}z^{(\varepsilon)}_i(0) -\frac{\lambda_i^2}{\alpha_i}(1-e^{-(t-s)\frac{\alpha_i}{\varepsilon}})\nabla K(p^{(\varepsilon)}(s)) \Big\rangle\, \df s\\
    \leq &  \int_0^t 2\Big|\frac{\epsilon p}{\sqrt{1+\epsilon|p^{(\varepsilon)}|^2}}U(q^{(\varepsilon)}(s))+p^{(\varepsilon)}(s)\Big|^2 +C|\nabla K (p^{(\varepsilon)}(s))|^2\,\df s + C\sum_{i=1}^M\left|z_i^{(\varepsilon)}(0)\right|^2 \\
    \leq &\int_0^t C\epsilon U(q^{(\varepsilon)}(s))^2 + C\left|p^{(\varepsilon)}(s)\right|^2 \,\df s+C\sum_{i=1}^M\left|z_i^{(\varepsilon)}(0)\right|^2.
\end{align*}
Furthermore, it holds that
\begin{align*}
\int_0^t\Big(\gamma+\sum_{i=1}^M\frac{\lambda_i^2}{\alpha_i}\left(1-e^{-(t-s)\frac{\alpha_i}{\varepsilon}}\right)^2\Big)\Big(\frac{\epsilon dU(q^{(\varepsilon)}(s))}{(1+\epsilon|p^{(\varepsilon)}(s)|^2)^{\frac{3}{2}}}+d\Big)\,\df s \leq \int_0^t C\sqrt{1+\epsilon|p^{(\varepsilon)}(s)|^2}U(q(s))\,\df s.
\end{align*}
Thus, we obtain
\begin{align*}
    \IE\Big[ \sup_{t\in{[0,T]}}\left|\Gamma(t)\right|^n\Big]\leq &\IE\left[\left| \Gamma(0)\right|^n \right]+ C\sum_{i=1}^M\left|z_i^{(\varepsilon)}(0)\right|^2+  \int_0^TC\IE\Big[ \sup_{s\in[0,t]} \left|\Gamma(s)\right|^n\Big]\,\df t \\+&C\IE\Big[\sup_{t\in[0,T]}\Big|\int_0^t\Big\langle \frac{\epsilon p^{(\varepsilon)}(s)U(q^{(\varepsilon)}(s))}{\sqrt{1+\epsilon|p^{(\varepsilon)}(s)|^2}} +p^{(\varepsilon)}(s),\df W(s) \Big\rangle \Big|^n \\ +    & C\sum_{i=1}^M \sup_{t\in[0,T]}\Big| \int_0^t  \left(1-e^{-(t-s)\frac{\alpha_i}{\varepsilon}}\right)\Big\langle \frac{\epsilon p^{(\varepsilon)}(s)U(q^{(\varepsilon)}(s))}{\sqrt{1+\epsilon|p^{(\varepsilon)}(s)|^2}} +p^{(\varepsilon)}(s),\df W_i(s) \Big\rangle\Big|^n\Big].
\end{align*}
Through application of the Burkholder-Davis-Gundy Inequality, for $j\in\left\{ 0,1 \right\}$, we have the following estimates
\begin{align*}
    &\IE\Big[\sup_{t\in[0,T]}\Big| \int_0^t  \Big(1-e^{-(t-s)\frac{\alpha_i}{\varepsilon}}\Big)^j\Big\langle \frac{\epsilon p^{(\varepsilon)}(s)U(q(^{(\varepsilon)}s))}{\sqrt{1+\epsilon|p^{(\varepsilon)}(s)|^2}} +p^{(\varepsilon)}(s),\df W(s) \Big\rangle\Big|^n\Big]\\
    \leq &C\IE\Big[ \Big(\int_0^T  \left(1-e^{-(t-s)\frac{\alpha_i}{\varepsilon}}\right)^{2j} \frac{\epsilon^2 |p^{(\varepsilon)}(s)|^2 U(q^{(\varepsilon)}(s))^2}{1+\epsilon |p^{(\varepsilon)}(s)|^2} +|p^{(\varepsilon)}(s)|^2 \,\df s\Big)^{\frac{n}{2}}
    \Big]\\
    \leq& C\IE \Big[ \Big(\int_0^T   \frac{\epsilon^2 |p^{(\varepsilon)}(s)|^2 U(q^{(\varepsilon)}(s))^2}{1+\epsilon |p^{(\varepsilon)}(s)|^2} +|p^{(\varepsilon)}(s)|^2 \,\df s\Big)^{n} \Big] +C\\
    \leq& C\int_0^T\IE \Big[   \sup_{s\in[0,t]} \epsilon^n U(q^{(\varepsilon)}(s))^{2n} \Big]\df t + C\int_0^T\IE \Big[\sup_{s\in[0,t]}\left|p^{(\varepsilon)}(s)\right|^{2n}\Big]\, \df t + C.
\end{align*}
Putting the above together, we get
\begin{align*}
     \IE\Big[ \sup_{t\in{[0,T]}}\left|\Gamma_1(t)\right|^n\Big]\leq  &\IE\left[\left| \Gamma_1(0)\right|^n \right]+C\sum_{i=1}^M\IE\left[\left|z_i^{(\varepsilon)}(0)\right|^{2n}\right] +\int_0^TC\IE\Big[ \sup_{s\in[0,t]} \left|\Gamma_1(s)\right|^n\Big]\,\df t+C.
\end{align*}
By applying Gronwall's inequality, we get
\begin{align}
\label{Gammabound}
     \IE\left[ \sup_{t\in{[0,T]}}\left|\Gamma_1(t)\right|^n\right]\leq  &\IE\left[\left| \Gamma_1(0)\right|^n \right] +C\sum_{i=1}^M\IE\left[\left|z_i^{(\varepsilon)}(0)\right|^{2n}\right] +C.
\end{align}
From Assumption \ref{boundedinitialconditionsrel}, $\IE\left[|\Gamma(0)|^n\right]<C$. Together with \eqref{Gammabound}, we find
\[\IE\left[\sup_{t\in[0,T]}|q(t)|^n+|p(t)|^n\right] \leq    C.\]
It follows from \eqref{zito} that, for $1\leq i\leq M$,
\begin{align}
\IE\left[\sup_{t\in[0,T]}|z^{(\varepsilon)}_i(t)|^n\right]\leq   C\left(\IE\left[|z_i^{(\varepsilon)}(0)|^n\right]+\IE\left[\sup_{t\in[0,T]}|p^{(\varepsilon)}(t)|^n\right]+C\right).
\end{align}
The result \eqref{GLREExpectationbound} clearly follows.

Concerning \eqref{GLEExpectationbound}-\eqref{RLEExpectationbound}, on the one hand, we note that the proof of \eqref{RLEExpectationbound} is the same as that of \cite[Lemma 4.4]{duong2024asymptotic}. On the other hand, estimate \eqref{GLEExpectationbound} can be derived using an argument to similar to the above. Inspired by of\cite[Proposition 13]{nguyen2018small}, we choose
\begin{align*}
    \Gamma_2(t)= \frac{1}{2}|p^{(\varepsilon)}(t)|^2+U(q^{(\varepsilon)}(t)).
\end{align*}
By applying Itö's formula,

\begin{align*}
    \Gamma_2(t)=\Gamma_2(0)
-& \int_0^t\left|p^{(\varepsilon)}(s)\right|^2\df s-\int_0^t\Big\langle p^{(\varepsilon)}(s),
     \sum_{i=1}^M \frac{\lambda_i}{\sqrt{\varepsilon}}e^{-\frac{\alpha_i}{\varepsilon}s}z^{(\varepsilon)}_i(0) -\frac{\lambda_i^2}{\alpha_i}(e^{-(t-s)\frac{\alpha_i}{\varepsilon}}+1)p^{(\varepsilon)}(s) \Big\rangle\, \df s\\
    +&\int_0^td\Big(\gamma+\sum_{i=1}^M\frac{\lambda_i^2}{\alpha_i}\left(1-e^{-(t-s)\frac{\alpha_i}{\varepsilon}}\right)^2\Big)\,\df s
    + \int_0^t\sqrt{2\gamma}\Big\langle p^{(\varepsilon)}(s),\df W(s) \Big\rangle\\
    +    &\sum_{i=1}^M\int_0^t  \sqrt{\frac{{2\lambda_i^2}}{{\alpha_i}}}\left   (1-e^{(t-s)\frac{\alpha_i}{\varepsilon}}\right)\Big\langle p^{(\varepsilon)}(s),\df W_i(s) \Big\rangle.
\end{align*}
By modifying the above proof to the classical setting, we obtain

\begin{align*}
     \IE\Big[ \sup_{t\in{[0,T]}}\left|\Gamma_2(t)\right|^n\Big]\leq  &\IE\left[\left| \Gamma_2(0)\right|^n \right]+C\sum_{i=1}^M\IE\left[\left|z_i^{(\varepsilon)}(0)\right|^{2n}\right] +\int_0^TC\IE\Big[ \sup_{s\in[0,t]} \left|\Gamma_2(s)\right|^n\Big]\,\df t+C.
\end{align*}
The result then follows using the Gronwall-type argument as above.

\end{proof}

\subsection{White-noise limit} \label{sec:white-noise-newtonian:white-noise}

Owing to the presence of the nonlinearity, we will not directly establish Theorem \ref{relativisticwhitenoiselimit} on the white noise limit of \eqref{GRLEpqzrescaled}. Instead, we will prove an analogue of Theorem \ref{relativisticwhitenoiselimit}  while assuming further that $\nabla U$ is Lipschitz. Then, we will remove such a restriction by leveraging the uniform moment bounds from Proposition \ref{boundedprocessesgen}. 

More specifically, we have the following auxiliary result, whose argument is similar to \cite[Theorem 2.6]{ottobre_asymptotic_2011} adapted to the relativistic setting.

\begin{proposition}
\label{relwhitenoiselimitlipschitz}
Let $\left(q^{(\varepsilon)},p^{(\varepsilon)},z^{(\varepsilon)}\right)$ be the solution to \eqref{GRLEpqzrescaled}. Then, under Assumptions \ref{UAssumptions} and \ref{boundedinitialconditionsrel}, and that $\nabla U$ is Lipschitz continuous, for all $T>0$ and $n\in\IN$, it holds that
\begin{align} \label{ineq:white-noise:Lipschitz}
    \IE\Big[\sup_{t\in[0,T]}|q^{(\varepsilon)}(t)-Q(t)|^n+ \sup_{t\in[0,T]}|p^{(\varepsilon)}(t)-P(t)|^n\Big]\leq C \varepsilon^\frac{n}{2},
\end{align}
for some constant $C=C(T,\lambda,\alpha,n,M)>0$. In the above, $(Q,P)$ is defined as in \eqref{PQprocess}.
\end{proposition}
\begin{proof}
   In the below, $C$ is independent of $c$ and $\varepsilon$ and may change from line to line. Recall that we assume $q^{(\varepsilon)}$ and $p^{(\varepsilon)}$ have the same initial conditions of $Q$ and $P$ respectively. Proceeding similarly to \cite{ottobre_asymptotic_2011}, from \eqref{qrescaled}, and the Lipschitz continuity of $\nabla K (p)$, we have
\begin{equation}
    \left|q^{(\varepsilon)}(t)-Q(t)\right|\leq C\int_0^t\left|p^{(\varepsilon)}(s)-P(s) \right|\df s.
\end{equation}
From \eqref{zrescaled}, we get
\begin{equation}
    \frac{1}{\sqrt{\varepsilon}}z^{(\varepsilon)}_i(s)\df s=-\frac{\sqrt{\varepsilon}}{\alpha_i}\left(z_i^{(\varepsilon)}(t)-z_i^{(\varepsilon)}(0)\right)-\frac{\lambda_i}{\alpha_i}\int_0^t\nabla K(p^{(\varepsilon)}(s)) \df s+\sqrt{\frac{2}{\alpha_i}}\text{d}W_i(t).
\end{equation}

\noindent Next from \eqref{prescaled} and \eqref{bigP}, we have
\begin{align*}
p^{(\varepsilon)}(t)-P(t)=&\int_0^t\left(\nabla U(Q(s))-\nabla U(q^{(\varepsilon)}(s))\right)\df s-\sum_{i=1}^M\sqrt{\varepsilon}\left(\frac{\lambda_i}{\alpha_i}\left(z_i^{(\varepsilon)}(t)-z_i^{(\varepsilon,}(0)\right)\right)\\
+&\Big(\gamma+\sum_{i=1}^M\frac{\lambda_i^2}{\alpha_i}\Big)\int_0^t\left(\nabla  K(P(s))-\nabla K(p^{(\varepsilon)}(s))\right)\df s.
\end{align*}
For notational convenience, we define
 \[\chi_1(T)=\IE\Big[\sup_{t\in[0,T]}\left(\left|q^{(\varepsilon)}(t)-Q(t)\right|^n+\left|p^{(\varepsilon)}(t)-P(t)\right|^n\right)\Big].\]
By the Lipschitz continuity of $\nabla U$ and H\"older's inequality, we have
\begin{align*}
    \chi_`(T)
    &\leq CT^{n-1}\int_0^T\IE\left[\sup_{s\in[0,t]}|q^{(\varepsilon)}(s)-Q(s)|^n\right]\df t \\
    &+C\left(\gamma+\sum_{i=1}^M\frac{\lambda_i^2}{\alpha_i}\right)T^{n-1}\int_0^T\IE\left[\sup_{s\in[0,t]}|p^{(\varepsilon)}(s)-P(s)|^n\right]\df t\\
    &+C\varepsilon^\frac{n}{2}\sum_{i=1}^M\left(\frac{\lambda_i}{\alpha_i}\right)^n\IE\left[\sup_{t\in[0,T]}|z_i^{(\varepsilon)}(t)-z_i^{(\varepsilon)}(0)|^n\right].
\end{align*}
From this we deduce
\begin{align*}
    \chi_1(T)\leq C\int_0^T\chi_1(t)\df t+C\varepsilon^{\frac{n}{2}}\sum_{i=1}^M\IE\left[\sup_{t\in[0,T]}|z_i^{(\varepsilon)}(t)-z_i^{(\varepsilon)}(0)|^n\right].
\end{align*}
By applying Gronwall's lemma, we obtain
\begin{align*}
    \chi_1(T)&\leq C\varepsilon^\frac{n}{2}\sum_{i=1}^M\IE\Big[\sup_{t\in[0,T]}|z_i^{(\varepsilon)}(t)-z_i^{(\varepsilon)}(0)|^n\Big]\\&+C\varepsilon^{\frac{n}{2}}\int_0^T\sum_{i=1}^M\IE\Big[\sup_{s\in[0,t]}|z_i^{(\varepsilon)}(s)-z_i^{(\varepsilon)}(0)|^n\Big]\df t,
\end{align*}
whence
\begin{align*}
    \chi_1(T)&\leq C\varepsilon^\frac{n}{2}\sum_{i=1}^M\IE\Big[\sup_{t\in[0,T]}|z_i^{(\varepsilon)}(t)-z_i^{(\varepsilon)}(0)|^n\Big].
\end{align*}
In light of Proposition \ref{boundedprocessesgen} together with \ref{boundedinitialconditionsrel}, the supremum on the above right hand side is uniformly bounded independent of $\varepsilon$. This establishes \eqref{ineq:white-noise:Lipschitz}, as claimed.

\end{proof}

We now present the proof of Theorem \ref{relativisticwhitenoiselimit}, whose argument will employ the results from Propositions \ref{boundedprocessesgen} and \ref{relwhitenoiselimitlipschitz}.

\begin{proof}[Proof of Theorem \ref{relativisticwhitenoiselimit}]
    For $R>0$, we define the following stopping times
\begin{align*}
   \tau_R =\inf\Big\{ T: \sup_{t\in[0,T]} |q^{(\varepsilon)}(t)|>R \Big\};
\end{align*}
and
\begin{align*}
    \tau^\varepsilon_R =\inf\Big\{ T: \sup_{t\in[0,T]} |Q(t)|>R \Big\}.
\end{align*}
We also define a cut-off function by
\[\theta_R(x)=\begin{cases}
    1, &|x|\leq R,\\
    0, &|x|\geq R+1.
\end{cases}\]
Using this cut-off function, we next introduce the cut-off system for the GRLE
 \begin{align*}
dq^{(R,\varepsilon)}(t)&=\nabla K (p^{(R,\varepsilon)}(t))\,\df t,\\
 \notag dp^{(R,\varepsilon)}(t)&=-\nabla \theta_R(q(t))U(q^{(R,\varepsilon)}(t))\,\df t-\gamma\nabla K (p^{(R,\varepsilon)}(t))\,\df t+\frac{1}{\sqrt{\varepsilon}}\sum_{i=1}^M \lambda_i z_i^{(R,\varepsilon)}(t)\,\df t\\ &+\sqrt{2\gamma}\, \df W(t),\\
 dz_i^{(R,\varepsilon)}(t)&=-\frac{\lambda_i}{\varepsilon}\nabla K (p^{(R,\varepsilon)}(t))\,\df t-\frac{\alpha_i}{\varepsilon} z_i^{(R,\varepsilon)}(t)\,\df t+\sqrt{\frac{2\alpha_i}{{\varepsilon}}}\,\df W_i(t),
\end{align*}
\[q^{(R,\varepsilon)}(0)= q_0,\qquad p^{(R,\varepsilon)}(0)=p_0,\qquad z_i^{(R,\varepsilon)}(0)=z_{i,0},\]
and the corresponding cut-off system for the relativistic underdamped Langevin dynamics: 
 \begin{align*}
\df Q^{(R)}(t)&=\nabla K (P^{(R)}(t))\,\df t,\\
 \df P^{(R)}(t)&=-\nabla U(Q^{(R)}(t))\,\df t-\left( \gamma+\sum_{i=1}^M\frac{\lambda^2_i}{\alpha_i} \right) \nabla K (P^{(R)}(t))\,\df t+\sqrt{2\beta^{-1}   \gamma} \, d W(t)+ \sum_{i=1}^M\sqrt{  \frac{2\beta^{-1} \lambda_i^2}{\alpha_i}} \, \df W_i(t),
\\ Q(0)&=q_0,\quad P(0)=p_0.
\end{align*}   
By conditioning on the events $\left\{ \tau_R<T  \right\}$ and $\left\{ \tau^\varepsilon_R<T  \right\}$ we have
     \begin{align*}
         &\IE\Big[\sup_{t\in[0,T]}\left|q^{(\varepsilon)}(t)-Q(t)\right|^n+\sup_{t\in[0,T]}\left|p^{(\varepsilon)}(t)-P(t)\right|^n\Big]\\
         \leq&\IE\left[\sup_{t\in[0,T]}\left|q^{(R,\varepsilon)}(t)-Q^{(R)}(t)\right|^n+\sup_{t\in[0,T]}\left|p^{(R,\varepsilon)}(t)-P^{(R)}(t)\right|^n\right]\\
         &+\IE\left[\left(\sup_{t\in[0,T]}\left|q^{(\varepsilon)}(t)-Q(t)\right|^n+\sup_{t\in[0,T]}\left|p^{(\varepsilon)}(t)-P(t)\right|^n\right)\mathbbm{1}\left\{\tau_R>T\right\}\right]\\
         &+\IE\left[\left(\sup_{t\in[0,T]}\left|q^{(\varepsilon)}(t)-Q(t)\right|^n+\sup_{t\in[0,T]}\left|p^{(\varepsilon)}(t)-P(t)\right|^n\right)\mathbbm{1}\left\{\tau_R^\epsilon>T\right\}\right].
    \end{align*}
    By Proposition \ref{relwhitenoiselimitlipschitz}, we have
    \[ \IE\left[\sup_{t\in[0,T]}\left|q^{(R,\varepsilon)}(t)-Q^{(R)}(t)\right|^n+\sup_{t\in[0,T]}\left|p^{(R,\varepsilon)}(t)-P^{(R)}(t)\right|^n\right]\leq C\varepsilon^{\frac{n}{2}}.\]
    We emphasize that the above constant $C=C(R)$ does not depend on $\varepsilon$. From Hölder's inequality, we can estimate
\begin{align*}
    &\IE\left[\left(\sup_{t\in[0,T]}\left|q^{(\varepsilon)}(t)-Q(t)\right|^n+\sup_{t\in[0,T]}\left|p^{(\varepsilon)}(t)-P(t)\right|^n\right)\mathbbm{1}\left\{\tau_R>T\right\}\right]\\
    \leq &\IE\left[\sup_{t\in[0,T]}\left|q^{(\varepsilon)}(t)-Q(t)\right|^{2n}+\sup_{t\in[0,T]}\left|p^{(\varepsilon)}(t)-P(t)\right|^{2n}\right]^{\frac{1}{2}}\IP\left[\tau_R>T\right]^\frac{1}{2}.
\end{align*}
By Markov's Inequality, and Proposition \ref{boundedprocessesgen}, we have

\[\IP\left[\tau_R>T\right]\leq\frac{\IE\left[\sup_{t\in[0,T]}|Q(t)|\right]}{R}\leq\frac{C}{R}\to0,\qquad R\to\infty.\]
By Proposition \ref{boundedprocessesgen}, we have

\begin{align*}
    &\IE\left[\sup_{t\in[0,T]}\left|q^{(\varepsilon)}(t)-Q(t)\right|^{2n}+\sup_{t\in[0,T]}\left|p^{(\varepsilon)}(t)-P(t)\right|^{2n}\right]^{\frac{1}{2}}\\
    \leq&\IE\left[\sup_{t\in[0,T]}\left|q^{(\varepsilon)}(t)\right|^{2n}+\sup_{t\in[0,T]}\left|Q(t)\right|^{2n}+\sup_{t\in[0,T]}\left|p^{(\varepsilon)}(t)\right|^{2n}+\sup_{t\in[0,T]}\left|P(t)\right|^{2n}\right]< C.
\end{align*}
Thus, it holds that
\begin{align*}
    \IE\Big[\Big(\sup_{t\in[0,T]}\left|q^{(\varepsilon)}(t)-Q(t)\right|^n+\sup_{t\in[0,T]}\left|p^{(\varepsilon)}(t)-P(t)\right|^n\Big) \mathbbm{1}\left\{\tau_R<T\right\}\Big]\le \frac{\tilde{C}}{\sqrt{R}},
\end{align*}
where the positive constant $\tilde{C}$ is independent of both $\varepsilon$ and $R$. By a similar argument, we also obtain
 \begin{align*}
     \IE\Big[\Big(\sup_{t\in[0,T]}\left|q^{(\varepsilon)}(t)-q(t)\right|^n+\sup_{t\in[0,T]}\left|p^{(\varepsilon)}(t)-p(t)\right|^n\Big)\mathbbm{1}\left\{\tau_R^\epsilon<T\right\}\Big]\le \frac{\tilde{C}}{\sqrt{R}}.
 \end{align*}
Altogether, we get
\begin{align*}
    &\IE\Big[\sup_{t\in[0,T]}\left|q^{(\varepsilon)}(t)-Q(t)\right|^n+\sup_{t\in[0,T]}\left|p^{(\varepsilon)}(t)-P(t)\right|^n\Big]\le C(R)\varepsilon^{\frac{n}{2}}+\frac{\tilde{C}}{\sqrt{R}}.
\end{align*}
The result follows by sending $R>0$ to infinity and then shrinking $\varepsilon$ to zero. The proof is thus finished.
\end{proof}
\subsection{Newtonian limit} \label{sec:white-noise-newtonian:newtonian}
In this section, we prove Theorem \ref{relativisticnewtonianlimit} giving the validity of the approximation of the GRLE \eqref{GRLEpqz} by the GLE \eqref{GLEpqz} in the Newtonian limit as the light speed parameter $c$ tends to infinity. Thanks to Proposition \ref{boundedprocessesgen}, the approach of Theorem \ref{relativisticnewtonianlimit} is similar to that of the white noise regime presented in Section \ref{sec:white-noise-newtonian:white-noise}. More specifically, we proceed to establish limit \eqref{lim:newtonian} under the extra condition that $\nabla U $ is Lipschitz. This is demonstrated through Proposition \ref{newtonianlimitlipschitz} below. Then, the uniformity with respect to $c$ on the energy estimates from Proposition \ref{boundedprocessesgen} will allow us to overcome the Lipschitz restriction, thereby recovering limit \eqref{lim:newtonian} for the general class of potentials $U$ as in Assumption \ref{UAssumptions}.

\begin{proposition}
\label{newtonianlimitlipschitz}
Let $(q^{(c)},p^{(c)},z^{(c)})$ be the solution to \eqref{GRLEpqz}. Then, under Assumptions \ref{UAssumptions} and \ref{boundedinitialconditionsrel}, and that $\nabla U$ is Lipschitz continuous, for all $T>0$ and $n\in\IN$, the following holds
\begin{align} \label{lim:newtonian:Lipschitz}
    \IE\left[\sup_{t\in[0,T]}\left|q^{(c)}(t)-q(t)\right|^n+\sup_{t\in[0,T]}\left|p^{(c)}(t)-p(t)\right|^n+\sum_{i=1}^M\sup_{t\in[0,T]}\left|z_i^{(c)}(t)-z_i(t)\right|^n\right]\leq C \epsilon^\frac{n}{2},
\end{align}
for some constant $C=C(U,T,\lambda,\alpha,n,M)>0$ independent of $\epsilon =1/c^2$. In the above, $(q,p,z)$ is the process solving the non-relativistic generalized Langevin dynamics \eqref{GLEpqz}.
\end{proposition}

\begin{proof}
In the below, $C$ is independent of $c$ and may change from line to line. Recall that we assume $q^{(c)}$, $p^{(c)}$ and $z^{(c)}$ have the same initial conditions of $q$, $p$ and $z$ respectively. Note that for all $p\in\IR^d$, we have
\begin{equation*}
   \left| \frac{p}{\sqrt{1+{\epsilon}|p|^2}}-p\right|=\frac{{\epsilon} |p|^3}{\left(1+\sqrt{1+\epsilon|p|^2}\right)\sqrt{1+{\epsilon}|p|^2}}\leq\sqrt{{\epsilon}}|p|^2.
\end{equation*}
By the triangle inequality, it follows that
\begin{align}
\label{differenceofrelandclas}
    |\nabla  K(p^{(c)}(t))-p(t)|
  &=  \Big|\frac{p^{(c)}(t)
}{\sqrt{1+\epsilon|p^{(c)}(t)|^2}}-p(t)\Big|\notag\\&\leq \Big|\frac{p^{(c)}(t)
}{\sqrt{1+\epsilon|p^{(c)}(t)|^2}}-p^{(c)}(t)\Big|+|p^{(c)}(t)-p(t)|\notag \\&\leq|p^{(c)}(t)-p(t)|+\sqrt{\epsilon}|p^{(c)}(t)|^2.
\end{align}
By \eqref{differenceofrelandclas}, \eqref{qrelmarkov} and \eqref{qclasmarkov}, we have
\begin{equation}
    |q^{(c)}(t)-q(t)|=\left|\int_0^t\nabla  K(p^{(c)}(s))-p(s)\,\df s\right| \leq C\int_0^t(\sqrt{{\epsilon}}|p^{(c)}(s)|^2+|p^{(c)}(s)-p(s)|)\, \df s.
\end{equation}
From \eqref{zclasmarkov} and \eqref{zrelmarkov}, we get
\begin{equation*}
    \int_0^t z_{k}^{(c)}(s)-z_k(s)\df s= \frac{1}{\alpha_i}(z_k(t)-z_k(0))-\frac{1}{\alpha_i}(z_K (t)-z_K (0))-\frac{\lambda_i}{\alpha_i}\int_0^t\left(\nabla K(p^{(c)}(s))-p(s)\right)\df s.
\end{equation*}
Also, from \eqref{pclasmarkov} and \eqref{prelmarkov}, we have
\begin{align}
        p^{(c)}(t)-p(t)=&\int_0^t\nabla U(q(s))-\nabla U(q^{(c)}(s))\,\df s+\int_0^tp(s)-\frac{p^{(c)}(s)}{\sqrt{1+{\epsilon} |p^{(c)}(s)|^2}}\,\df t\notag\\&+\sum_{i=1}^M\lambda_i \int_0^t z_{i}^{(c)}(s)-z_i(s)\,\df s.
\end{align}
We define
\[
\chi_2(T)=\IE\Big[\sup_{t\in[0,T]}\Big(\left|q^{(c)}(t)-q(t)\right|^n+\left|p^{(c)}(t)-p(t)\right|^n+\sum_{i=1}^M\left|z^{(c)}(t)-z(t)\right|^n\Big)\Big].
\]
Applying the Lipschitz continuity of $\nabla K $ and $\nabla U$, as well as \eqref{differenceofrelandclas}, we can estimate $\chi_2(T)$ as follows:
\begin{align*}
\chi_2(T)&\leq \int_0^T\IE\Big[\sup_{s\in[0,t]}|q^{(c)}(s)-q(s)|^n\Big]\df t \\
    &+C\Big(1+\sum_{i=1}^M\frac{\lambda_i^2}{\alpha_i}\Big)T^{n-1}\int_0^T\IE\Big[\sup_{s\in[0,t]}|p^{(c)}(s)-p(s)|^n\Big]\df t\\
    &+C{\epsilon}^\frac{n}{2}\Big(1+\sum_{i=1}^M\frac{\lambda_i^2}{\alpha_i}\Big)T^{n-1}\int_0^T\IE\Big[\sup_{s\in[0,t]}|p^{(c)}(s)|^{2n}\Big]\df t\\
    &+C{\epsilon}^\frac{n}{2}\sum_{i=1}^M\left(\frac{\lambda_i}{\alpha_i}\right)^n\IE\Big[\sup_{t\in[0,T]}|z^{(c)}_i(t)-z^{(c)}_i(0)|^n\Big]\\
     &+C{\epsilon}^\frac{n}{2}\sum_{i=1}^M\left(\frac{\lambda_i}{\alpha_i}\right)^n\IE\Big[\sup_{t\in[0,T]}|z_i(t)-z_i(0)|^n\Big]\\
     &+\sum_{i=1}^M\alpha_iT^{n-1}\int_0^T\IE\Big[\sup_{s\in[0,t]}|z_i^{(c)}(s)-z_i(s)|^n\Big]\df s.
\end{align*}
In view of Proposition \ref{boundedprocessesgen}, we deduce that
\begin{align*}
    \chi_2(T)\leq &C\int_0^T\chi_2(t)\,\df t\\+&C{\epsilon}^{\frac{n}{2}}\Big(\sum_{i=1}^M\IE\Big[\sup_{t\in[0,T]}|z_i^{(c)}(t)-z_i^{(c)}(0)|^n\Big]+\IE\Big[\sup_{t\in[0,T]}|z_i(t)-z_i(0)|^n\Big]+1\Big).
\end{align*}
An application of Gronwall's lemma produces the bound
\begin{align*}
    \chi_2(T)&\leq C{\epsilon}^\frac{n}{2}\Big(\sum_{i=1}^M\IE\Big[\sup_{t\in[0,T]}|z_i^{(c)}(t)-z_i^{(c)}(0)|^n\Big]+\sum_{i=1}^M\IE\Big[\sup_{t\in[0,T]}|z_i(t)-z_i(0)|^n\Big]+1\Big)\\&+C{\epsilon}^{\frac{n}{2}}\int_0^T\Big(\sum_{i=1}^M\IE\Big[\sup_{s\in[0,t]}|z_i^{(c)}(s)-z_i^{(c)}(0)|^n\Big]+\sum_{i=1}^M\IE\Big[\sup_{s\in[0,t]}|z_i(s)-z_i(0)|^n\Big]+1\Big)\df t,
\end{align*}
whence
\begin{align*}
    \chi_2(T)&\leq C{\epsilon}^\frac{n}{2}\Big(\sum_{i=1}^M\IE\Big[\sup_{t\in[0,T]}|z_i^{(c)}(t)-z_i^{(c)}(0)|^n\Big]+\sum_{i=1}^M\IE\Big[\sup_{t\in[0,T]}|z_i(t)-z_i(0)|^n\Big]+1\Big),
\end{align*}
where $C=C(T)$ is independent of $\epsilon$. We invoke Proposition \ref{boundedprocessesgen} once again to deduce estimate \eqref{lim:newtonian:Lipschitz}, as claimed.
\end{proof}

Lastly, let us conclude Theorem \ref{relativisticnewtonianlimit} by combining the auxiliary results from Propositions \ref{boundedprocessesgen} and \ref{newtonianlimitlipschitz}.

\begin{proof}[Proof of Theorem \ref{relativisticnewtonianlimit}]
    Through applications of Propositions \ref{boundedprocessesgen}, \ref{newtonianlimitlipschitz}, the proof of Theorem \ref{relativisticnewtonianlimit} is can be carried out by making use of suitable stopping times of exiting bounded sets. Since the argument is similar to that of Theorem \ref{relativisticwhitenoiselimit}, its detail is thus omitted.
\end{proof}

\appendix

\section{Ergodicity for a general SDE}
\label{Ergodicity Results}
\label{proofofpolynomialmixing}
In this appendix, we briefly review the framework of \cite{hairer2009hot} for obtaining rate of convergence to the equilibrium for a general stochastic differential equation. We have employed this approach to establish Theorem \ref{polynomialerogicity} giving polynomial mixing for the GRLE \eqref{GRLEpqz}.

We denote $X$ to be the solution of the Itô SDE:
\begin{equation}
\label{general SDE}
\df X(t)=a(X(t))\,\df t+\sigma(X(t))\, \df W(t),\quad X(0)=x_0\in\IR^d.
\end{equation}
Let $\CL$ denote the infinitesimal generator for $X$. Following \cite[Section 3.2]{hairer2009hot}, we recall the notion of $\phi$-Lyapunov functions: 
\begin{definition}
\label{philyapunov}
    Given $\phi:\IR^+\to\IR^+$, a function $V\in C^2(\IR^d;\IR^+)$ is called \emph{$\phi$-Lyapunov} if
\[\CL V(x)\leq -\phi\left(V(x)\right),\]
outside some compact set, and $\lim_{|x|\to\infty}V(x)=\infty$.

\end{definition}
We introduce the notation:
\begin{align*}
    X_0=\sum_{i=1}^da_i(x)\frac{\delta}{\delta x_i},\quad
    X_k=\sum_{i=1}^k \sigma_{ik}(x)\frac{\delta}{\delta x_i}.
\end{align*}
We also define the vector fields recursively as:
\begin{align*}
    \mathcal{A}_0=\Lie\left(X_i\right)_{i=1,...,d},\quad
    \mathcal{A}_n=\Lie\left([X_0,X]\right)_{X\in \mathcal{A}_{n-1}},
\end{align*}
and define
\[\mathcal{A}=\Lie(\mathcal{A}_n)_{n\in\IN_0}.\]
\begin{assumption}[\cite{pavliotis_stochastic_2014}, Assumption 6.3]
\label{Hörmandersconditionplain}
For all $x\in\IR^{d}$, the following identity holds
\[\text{span}\left\{X(x):X\in\mathcal{A} \right\}=\IR^{d}.\]
\end{assumption}
In the literature, Assumption \ref{Hörmandersconditionplain} is referred to as Hörmander's condition, which guarantees that the solution to \eqref{general SDE}  admits a smooth probability density function.

\begin{assumption}
\label{controlproblemplain}
For any initial condition $X(0)\in\IR^{d}$, there exists a path $U$, which satisfies the ordinary differential equation,
        \[\df X(t)=a(X(t))\,\df t+\sigma(X(t))\, \df U(t).\]
\end{assumption}
Assumption \ref{controlproblemplain} ensures that the origin is reachable given any starting point.

\begin{theorem}\cite[Theorem 3.5]{hairer2009hot}
\label{hairerbound}
Suppose there exists an increasing, smooth, strictly sublinear $\phi$-Lyapunov function $V$, and that Assumptions \ref{Hörmandersconditionplain} and \ref{controlproblemplain} hold. Let $(P_t(x,\cdot))_{t\geq0}$ denote the transition probability of $X$ given in \eqref{general SDE}. Then there exist a unique invariant measure $\pi$ and constant $\zeta>0$ such that

    \[\|P_t(x,\cdot)-\pi\|_{\text{TV}}\leq \zeta_0 V(x)\psi(t),\]
where
\[
\psi(t)=\frac{1}{\phi\ \circ H^{-1}_{\phi}(t)},\quad\text{and}\quad H_{\phi}(t)=\int_1^t\frac{1}{\phi(s)}\,\df s.
\]
\end{theorem}

\section{A Non-Singular Matrix}
\label{appendix2}
In this appendix, through Lemma \ref{nonsingularmatrix} below, we provide sufficient conditions guaranteeing the non-singularity of certain matrices. The result of which is used in the proof of Propositions \ref{Hörmanderscondition} and \ref{controlproblemstatement} that are employed to construct Lyapunov functions for \eqref{GRLEpqz}.

\begin{lemma}
\label{nonsingularmatrix}
    Let ${\bf{B}}=(a_{ij})_{1\leq i,j\leq n}$ be a $n \times n$ matrix be given in the form,
\[{\bf{A}}=\begin{pmatrix}
    \lambda+b_1^2 & b_1b_2 & \cdots &b_1b_n\\
    b_1b_2 & \lambda+b_2^2 & \cdots&b_2b_n\\
    \vdots & \vdots &\ddots &\vdots\\
    b_1b_n& b_2b_n & \cdots &\lambda+b_n^2
\end{pmatrix}\]
where $b=(b_1,\dots,b_n)\in\IR^n$ and $\lambda\in\IR \setminus \{0\}$. Suppose further that 
\begin{equation}
\label{matrixcondition}
   \lambda+\sum_{i=1}^nb_i^2\neq0. 
\end{equation}
Then $\bf{B}$ is non-singular.
\end{lemma}
\begin{proof}
We decompose $\bf{B}$ as follows:

\[{\bf{B}}=\lambda{ {{\bf{I}}} }_n + \begin{pmatrix}
    b_1 &0 &\cdots &0\\
    0 &b_2& \cdots &0\\
    \vdots &\vdots &\ddots &\vdots\\
    0 & 0 &\cdots &b_n
\end{pmatrix}
    \begin{pmatrix}
    b_1 &b_2 &\cdots &b_n\\
    b_1 &b_2& \cdots &b_n\\
    \vdots &\vdots &\ddots &\vdots\\
    b_1 & b_2&\cdots &b_n
\end{pmatrix}.
\]
By applying the Weinstein-Aronszajn identity (see \cite[p.271]{pozrikidis_introduction_2014}), we have

\[\det {\bf{B}}= \det {{\bf{B}}}_1\]
where

\[{\bf{B}}_1=
    \begin{pmatrix}
        \lambda+b_1^2 & b_2^2 &\cdots &b_n^2\\
        b_1^2 & \lambda+b_2^2 &\cdots &b_n^2\\
        \vdots & \vdots & \ddots &\vdots\\
        b_1^2 & b_2^2 & \cdots &\lambda +b_n^2
    \end{pmatrix}.
\]
By subtracting the first row from the other rows, we also have
\[\det {\bf{B}}= \det {{\bf{B}}}_2\]
where
\begin{align}
\label{determinantinequality}
   \bf{B}_2=\begin{pmatrix}
    \lambda +b_1^2& b_2^2 & \cdots & b_n^2\\
    -\lambda & \lambda & \cdots & 0\\
    \vdots & \vdots & \ddots & \vdots \\
   -\lambda & 0&\cdots & \lambda
\end{pmatrix}. 
\end{align}
We claim that by virtue of \eqref{matrixcondition}, ${\bf{B}}_2$ is non-singular. Indeed, observe that if $a\in\IR^n$ satisfies \[{\bf{B}}_2a=0,\]
then, from the the second row to the $n$-th row of ${\bf{B}}_2$, we get $a_1=...=a_n$. It follows that the first row, together with \eqref{matrixcondition}, implies that $a_1=...=a_n=0$. As a consequence, $\bf{B}_2$ must be non-singular, and thus so is $\bf{B}$. This completes the proof.
\end{proof}

\section*{Acknowledgement} The research of M. H. Duong was supported by EPSRC Grant EP/Y008561/1.
\sloppy
\printbibliography

\end{document}

%% file: diagram.tex
\tikzset{every picture/.style={line width=0.75pt}} 

\begin{tikzpicture}[x=0.75pt,y=0.75pt,yscale=-1,xscale=1]

\draw   (101.67,188.67) -- (211.58,188.67) -- (211.58,288.58) -- (101.67,288.58) -- cycle ;

\draw   (448.66,189.17) -- (558.57,189.17) -- (558.57,289.08) -- (448.66,289.08) -- cycle ;

\draw   (101.17,349.42) -- (211.08,349.42) -- (211.08,449.33) -- (101.17,449.33) -- cycle ;
\draw   (448.66,349.42) -- (558.57,349.42) -- (558.57,449.33) -- (448.66,449.33) -- cycle ;

\draw    (211,241) -- (447,241) ;
\draw [shift={(449,241)}, rotate = 180] [color={rgb, 255:red, 0; green, 0; blue, 0 }  ][line width=0.75]    (10.93,-3.29) .. controls (6.95,-1.4) and (3.31,-0.3) .. (0,0) .. controls (3.31,0.3) and (6.95,1.4) .. (10.93,3.29)   ;
\draw    (211.5,400.5) -- (447,400.5) ;
\draw [shift={(449,400.5)}, rotate = 180] [color={rgb, 255:red, 0; green, 0; blue, 0 }  ][line width=0.75]    (10.93,-3.29) .. controls (6.95,-1.4) and (3.31,-0.3) .. (0,0) .. controls (3.31,0.3) and (6.95,1.4) .. (10.93,3.29)   ;
\draw    (155.5,289) -- (155.5,347) ;
\draw [shift={(155.5,349)}, rotate = 270] [color={rgb, 255:red, 0; green, 0; blue, 0 }  ][line width=0.75]    (10.93,-3.29) .. controls (6.95,-1.4) and (3.31,-0.3) .. (0,0) .. controls (3.31,0.3) and (6.95,1.4) .. (10.93,3.29)   ;
\draw    (507.5,289) -- (507.5,347) ;
\draw [shift={(507.5,349)}, rotate = 270] [color={rgb, 255:red, 0; green, 0; blue, 0 }  ][line width=0.75]    (10.93,-3.29) .. controls (6.95,-1.4) and (3.31,-0.3) .. (0,0) .. controls (3.31,0.3) and (6.95,1.4) .. (10.93,3.29)   ;
\draw  [draw opacity=0] (1,170) -- (658,170) -- (658,469) -- (1,469) -- cycle ;

\draw (466.61,209.62) node [anchor=north west][inner sep=0.75pt]   [align=left] {\begin{minipage}[lt]{51.47pt}\setlength\topsep{0pt}
\begin{center}
Relativistic\\Langevin\\Equation
\end{center}

\end{minipage}};
\draw (472.11,380.38) node [anchor=north west][inner sep=0.75pt]   [align=left] {\begin{minipage}[lt]{44.13pt}\setlength\topsep{0pt}
\begin{center}
Langevin\\Equation
\end{center}

\end{minipage}};
\draw (114.63,369.88) node [anchor=north west][inner sep=0.75pt]   [align=left] {\begin{minipage}[lt]{57.73pt}\setlength\topsep{0pt}
\begin{center}
Generalized\\Langevin\\Equation
\end{center}

\end{minipage}};
\draw (102,310) node [anchor=north west][inner sep=0.75pt]   [align=left] {$\displaystyle c\rightarrow \infty $};
\draw (454.5,310) node [anchor=north west][inner sep=0.75pt]   [align=left] {$\displaystyle c\rightarrow \infty $};
\draw (314.75,213.5) node [anchor=north west][inner sep=0.75pt]   [align=left] {$\displaystyle \varepsilon \rightarrow 0$};
\draw (314.75,373.5) node [anchor=north west][inner sep=0.75pt]   [align=left] {$\displaystyle \varepsilon \rightarrow 0$};
\draw (165.5,311) node [anchor=north west][inner sep=0.75pt]   [align=left] {Theorem \ref{relativisticnewtonianlimit}};
\draw (518,311) node [anchor=north west][inner sep=0.75pt]   [align=left] {\cite[Theorem 2.6]{duong_trend_2024}};
\draw (280.75,252) node [anchor=north west][inner sep=0.75pt]   [align=left] {Theorem \ref{relativisticwhitenoiselimit}


};
\draw (270.75,413) node [anchor=north west][inner sep=0.75pt]   [align=left] {\cite[Theorem 2.6]{ottobre_asymptotic_2011}
};
\draw (115.13,198.62) node [anchor=north west][inner sep=0.75pt]   [align=left] {\begin{minipage}[lt]{57.73pt}\setlength\topsep{0pt}
\begin{center}
Generalized\\Relativistic\\Langevin\\Equation
\end{center}

\end{minipage}};

\end{tikzpicture}